\documentclass{amsart}
\usepackage{latexsym}
\usepackage{amssymb}
\usepackage{amsmath}
\usepackage{amsfonts}

\def\f{\textbf{f}}

\def\R{\mathbb{R}}

\newcommand{\bmat}{\left[\begin{matrix}}
\newcommand{\emat}{\end{matrix}\right]}
\usepackage[latin1]{inputenc} 
\usepackage{amsthm}
\usepackage{amsxtra}
\usepackage{amscd}
\usepackage{setspace}   
\usepackage{mathrsfs}   

\newtheorem{theorem}{Theorem}
\newtheorem{proposition}[theorem]{Proposition}
\newtheorem{lemma}[theorem]{Lemma}

\newtheorem*{corollary}{Corollary}
\numberwithin{theorem}{section}
\numberwithin{equation}{section}

\theoremstyle{remark}
\newtheorem*{remark}{Remark}

\theoremstyle{definition}

\newcommand{\Z}{\mathbb{Z}}
\newcommand{\Q}{\mathbb{Q}}
\newcommand{\C}{\mathbb{C}}

\newcommand{\spn}{\mathrm{span}}

\newcommand{\A}{\mathbb{A}}
\newcommand{\GL}{\mathrm{GL}}
\newcommand{\SL}{\mathrm{SL}}
\newcommand{\rk}{\mathrm{rk}}

\newcommand{\gp}{\mathfrak{p}} 
\newcommand{\cO}{\mathcal{O}}
\newcommand{\vol}{\mathrm{vol}}
\newcommand{\diag}{\mathrm{diag}}
\newcommand{\tr}{\mathrm{tr}}
\newcommand{\ord}{\mathrm{ord}}
\newcommand{\Log}{\mathrm{Log}}
\usepackage{enumerate}

\newcommand{\F}{\mathbb{F}}
\usepackage{tikz-cd}

    \title{Adelic Rogers integral formula}
    
    \author{Seungki Kim}

\newcommand{\Addresses}{{
  \bigskip
  \footnotesize

  Seungki Kim, \textsc{Department of Mathematical Sciences, University of Cincinnati, Cincinnati OH 40503}\par\nopagebreak
  \textit{E-mail address}: \texttt{seungki.math@gmail.com}

}}

\begin{document}

\maketitle

\begin{abstract}
We formulate and prove the extension of the Rogers integral formula (\cite{Rog55}) to the adeles of number fields. We also prove the second moment formulas for a few important cases, enabling a number of classical and recent applications of the formula to extend immediately to any number field.
\end{abstract}

\section{Introduction}

\subsection{History and motivation}

The Rogers integral formula \cite{Rog55} (also see Theorem \ref{thm:classic} below) is one of the main tools in the geometry of numbers, that allows one to study the statistical properties of the points of a random lattice. It is the natural generalization of the celebrated Siegel integral formula \cite{Sie45}, which Siegel stated at the end of the paper without proof.

\begin{theorem}[Rogers \cite{Rog55}, Siegel \cite{Sie45} for case $k=1$] \label{thm:classic}
Let $k < n$ be positive integers, and $f:(\R^n)^k \rightarrow \R$ be a Borel integrable function. Let $X_n = \SL(n,\Z) \backslash \SL(n,\R)$ be the moduli space of the lattices in $\R^n$ of covolume $1$, and $\mu$ be the unique right $\SL(n,\R)$-invariant probability measure on $X_n$. Then
\begin{equation} \label{eq:classic}
\int_{X_n} \sum_{x_1, \ldots, x_k \in \Z^n \atop \mathrm{indep.}} f(x_1g, \ldots, x_kg) d\mu(g) = \int_{(\R^n)^k} f(x_1, \ldots, x_k)dx_1\ldots dx_k.
\end{equation}
In addition, let us interpret $(\R^n)^k = \mathrm{Mat}_{k \times n}(\R)$, the set of all $k \times n$ matrices over $\R$. Then
\begin{align*}
& \int_{X_n} \sum_{X \in \mathrm{Mat}_{k \times n}(\Z) \atop \mathrm{no\,zero\,rows}} f(Xg) d\mu(g) \\
=& \int_{(\R^n)^k} f(X)dX + \sum_{m=1}^{k-1} \sum_D N(D)^n \int_{(\R^n)^m} f(D^{\tr}X) dX,
\end{align*}
where the sum over $D$ is over all $m \times k$ row-reduced echelon forms over $\Q$ of rank $m$, $N(D)$ is the density of the vectors $x \in \Z^m$ such that $xD \in \Z^k$.
\end{theorem}

Since their development in the mid-twentieth century, these formulas of Siegel and Rogers --- called the \emph{mean value formulas} in the literature --- have been applied extensively to the study of lattice problems, such as the packing and covering of $\R^n$ by spheres, and they still stand as one of the most powerful tools for such problems. Indeed, Davenport-Rogers \cite{DR47} had provided the best known lower bound $\approx 1.68 n2^{-n}$ for the sphere packing density in dimension $n$, until Ball \cite{Bal92} improved it to $\approx 2n2^{-n}$ several decades later. After a few more decades, Venkatesh \cite{Ven13} adopted the method of Rogers \cite{Rog47} to further improve it to $\approx 65963n2^{-n}$ for all sufficiently large $n$, which is the current best record. Moreover, Rogers \cite{Rog59} remains the best known bound on the covering density by a ball to this day, although for the covering by the general convex bodies, there is a recent substantial improvement by Ordentlich-Regev-Weiss \cite{ORV20}.

One also considers the distributional variants of these lattice problems, for instance, questions such as ``What is the $\mu$-measure of lattices that yield at least this much packing density?'' Investigating problems of this kind relies heavily, if not entirely, on the Rogers integral formula. The quoted problem above, in particular, had been intensively studied in a series of works by Rogers (see e.g. \cite{Rog55-1}, \cite{Rog56}) and Schmidt (e.g. \cite{Sch58}, \cite{Sch59}; also see the references in \cite{Sch59}). Recently, their ideas and methods have been further developed by the author (\cite{Kim16}, \cite{Kim20}), S\"odergren (e.g. \cite{Sod11}, \cite{Sod11-1}) and Str\"ombergsson-S\"odergren \cite{SS19}, with applications to the study of the Epstein zeta functions (\cite{Sod13}, \cite{SS18}).

Perhaps surprisingly, nowadays the mean value formulas have become a standard part of arsenal for homogeneous dynamics as well. Since the late 1990's to the present day, they have been employed in countless works in the field for various purposes, including a number of influential works, such as Eskin-Margulis-Mozes \cite{EMM98} on the quantitative Oppenheim conjecture, Kleinbock-Margulis \cite{KM99} and Athreya-Margulis \cite{AM09} on one-parameter flows, Markl\"of-Str\"ombergsson \cite{MS10} and Athreya-Margulis \cite{AM15} on counting problems, and so on. To quickly give one concrete application among many, the Rogers formula yields the upper bound portion of a logarithm law for free, that is, a statement of the form
\begin{equation*}
\limsup_{t \rightarrow \infty} \frac{\Delta(\Lambda g_t)}{\log t} = \alpha,
\end{equation*}
where $\Delta$ is some function on $X_n$, $\Lambda$ is a generic element of $X_n$, $(g_t)_{t \in \R}$ is a one-parameter subgroup of $\SL(n,\R)$, and $\alpha$ is a constant depending only on $\Delta$ and $n$ --- see \cite{KM99}, \cite{AM09}, \cite{KS21} for instance.

Naturally, there have been numerous efforts to extend the mean value formulas to a variety of contexts. There are variants that impose various different conditions on the sum on the left-hand side of \eqref{eq:classic}, such as summing over the primitive elements of $\Z^n$ --- see Section 1 of Schmidt \cite{Sch59} for examples. More recently, there are versions for the rational points on Grassmannians \cite{Kim22}, translation surfaces (\cite{Vee98}, see also related \cite{ACM19}), and for cut-and-project sets \cite{RSW20}, the latter two proved by ergodic theoretic methods. There are also extensions to $S$-unit lattices \cite{Han19}, to $\mathrm{ASL}(n,\R)$ (\cite{EMV15}, \cite{AGH21}), and to sums over translates of $\Z^n$ (\cite{GKY??}, \cite{AGH21}), which require one to consider the quotient of $\SL(n,\R)$ with respect to a congruence subgroup.

\subsection{The main theorem}

The goal of the present paper is to prove the following theorem, which extends the Rogers integral formula to the adele of a number field.

\begin{theorem}\label{thm:intro_main}
Let $F$ be a number field, and $\mathcal X_n$ be $\GL(n,F) \backslash G_n$, where
\begin{equation*}
G_n = \{A \in \GL(n, \A_F) : \|\det A\|_{\A_F} = 1 \}.
\end{equation*}
Also let $1 \leq k < n$. Then for a Borel integrable function $f: (\A_F^n)^k \rightarrow \R$, we have
\begin{equation*}
\int_{\mathcal X_n} \sum_{X \in (F^n)^k \atop \mathrm{rows\, indep.}} f(Xg) d\mu = \int_{(\A_F^n)^k} f(X) d\alpha_F^{nk},
\end{equation*}
where $\mu$ is the unique right $G_n$-invariant probability measure on $\mathcal X_n$, and $\alpha_F$ is the Tamagawa measure on $\A_F$, the Haar measure on $\A_F$ normalized so as to descend to the Haar probability measure on $\A_F / F$.

In addition, we have
\begin{align*}
&\int_{\mathcal X_n} \sum_{X \in (F^n)^k \atop \mathrm{no\, zero\, rows}} f(Xg) d\mu \\
=&\int_{(\A_F^n)^k} f(X) d\alpha_F^{nk} + \sum_{m=1}^{k-1} \sum_D \int_{(\A_F^n)^m} f(D^{\tr}X) d\alpha_F^{nm},
\end{align*}
where the sum over $D$ is over all $m \times k$ row-reduced echelon forms over $F$ of rank $m$.
\end{theorem}

We note that the case $k=1$ has been proved by Thunder \cite[Lemma 1]{Thu96}, who applies it to refine Siegel's lemma. Also see Venkatesh \cite[Theorem 1]{Ven13}, where the case $k=1$ and of level $1$ for the cyclotomic fields is proven in the classical language.


The main idea of our proof of Theorem \ref{thm:intro_main} is inspired by the original proof by Rogers of his formula \cite{Rog55}, although his argument contains a gap. Rogers argues that the left-hand side of \eqref{eq:classic} is equal to the limit of the integral of the sum
\begin{equation*}
 \sum_{x_1, \ldots, x_k \in \Z^n \atop \mathrm{indep.}} f(x_1g, \ldots, x_kg)
\end{equation*}
along a certain family of unipotent orbits in $\SL(n,\R)$, which in turn is equal to the Lebesgue measure of $f$. We now know both claims to be true by hindsight, but in proving the former he cites an incorrect statement about the fundamental domain of $X_n$ (\cite{Rog55}, p. 256, lines 2-8). We circumvent this problem by replacing the integration over unipotent orbits by the estimates on certain Hecke operators. This idea first appeared in the author's previous work \cite{Kim22}, in which it was implemented in the classical language.

By this method we obtain the so-called ``level 1'' case of the adelic Rogers formula (Theorem \ref{thm:main} below). To extend it to all levels and thereby fully prove Theorem \ref{thm:intro_main}, it turns out that we need the adelic version of the primitive analogue of Theorem \ref{thm:classic}, i.e. the statement
\begin{equation*}
\int_{X_n} \sum_{X \in (\Z^n)^k \atop \mbox{\tiny completes to $\SL(n,\Z)$}} f(Xg) d\mu(g) =\left( \prod_{j=n-k+1}^n \zeta^{-1}(j) \right)\int_{(\R^n)^k} f(X)dX,
\end{equation*}
a result due to Schmidt \cite{Sch57}. Such a statement (Theorems \ref{thm:meanhecke} and \ref{thm:prim}) is derived from the level 1 case by an inclusion-exclusion argument, which again involves the Hecke operators. The ``higher level'' version (Theorem \ref{thm:cong}) follows from this by a simple folding-unfolding trick on the sum over $(F^n)^m$, and subsequently an approximation argument yields Theorem \ref{thm:intro_main}.

It is easy to translate Theorem \ref{thm:intro_main} into the classical language, and unsurprisingly, certain recently proved variants of the Rogers formula (e.g. \cite{Han19}, \cite{AGH21}) follow immediately this way. Let us also demonstrate a quick example involving a number field other than $\Q$:
\begin{corollary} 
Let $F$ be a number field of class number $1$, and $X_n(F) = \SL(n,\cO_F) \backslash \SL(n,F \otimes \R)$, equipped with the probability measure $\mu$. Let $\lambda$ be the Lebesgue measure on $F \otimes \R$ normalized to assign covolume $1$ to the lattice obtained by the natural embedding $\cO_F \hookrightarrow F \otimes \R$. Then, for a Borel integrable $f_\infty$ on $\mathrm{Mat}_{k \times n}(F \otimes \R)$,
\begin{equation*}
\int_{X_n(F)} \sum_{X \in (\cO_F^n)^k \atop \mathrm{rows\,indep.}\, /F} f_\infty(Xg) d\mu \\
=\int_{(F\otimes\R)^{nk}} f_\infty(X) d\lambda^{nk},
\end{equation*}
and
\begin{align*}
&\int_{X_n(F)} \sum_{X \in (\cO_F^n)^k \atop \mathrm{no\, zero\, rows}} f_\infty(Xg) d\mu \\
=&\int_{(F\otimes\R)^{nk}} f_\infty(X) d\lambda^{nk} + \sum_{m=1}^{k-1} \sum_D N(D)^n \int_{(F\otimes\R)^{nm}} f_\infty(D^{\tr}X) d\lambda^{nm},
\end{align*}
where the sum over $D$ is over all $m \times k$ row-reduced echelon forms over $F$ of rank $m$, and $N(D)$ is the density of the vectors $x \in \cO_F^m$ such that $xD \in \cO_F^k$.
\end{corollary}
This follows from Theorem \ref{thm:intro_main} simply by taking $f = \mathbf{1}_{\cO_\f^{nk}}f_\infty$ and noting that
\[\GL(n,F) \backslash G_n / K \cong X_n(F),\] where $\cO_\f =  \prod_{\nu \nmid \infty} \cO_\nu$ and $K = \prod_{\nu \nmid \infty} \GL(n,\cO_\nu)$ (see e.g. \cite[Appendix A.3]{Gar90}). The class number $1$ condition can be dropped, by taking an appropriate adaptation of the proof of Theorem \ref{thm:main}; moreover, we can average over any rank $n$ torsion-free $\cO_F$-submodules of $(F\otimes\R)^n$ of a given Steinitz class, and obtain the same formula. For general $F$, $\GL(n,F) \backslash G_n / K$ has $\mathrm{Cl}(F)$ connected components, each being a moduli space of the $\cO_F$-modules of each Steinitz class; our proof of Theorem \ref{thm:main} carries over if we restrict $\mathcal X_n$ to any of these components.

Another point worth noting is that
\begin{equation*}
N(D)^n = \int_{\A_\f^{nm}} \mathbf{1}_{\cO_\f^{nk}}(D^\tr X) d\alpha_\f^{nm},
\end{equation*}
where $\A_\f$ is the finite part of $\A_F$ and $\alpha_\f$ is the Haar measure on $\A_\f$ that assigns measure $1$ to $\cO_\f$. By replacing $\mathbf{1}_{\cO_\f^{nk}}$ with other suitable function, one can obtain the congruence subgroup variant of the above corollary.

\subsection{Second moment estimates}

Even the simplest case $k=2$ of the classical Rogers integral formula alone has proved to be tremendously useful in the literature, through the second moment estimate of the lattice-point counting functions that follows from it. We present an adelic analogue of such, as an application of Theorem \ref{thm:intro_main}.

\begin{theorem} \label{thm:intro_second}
Let $n \geq 3$, and suppose $f:\A_F^n \rightarrow \R$ is a nonnegative function of the form $f_\f f_\infty$, where $f_\f$ is the characteristic function of $\prod_{\nu \nmid \infty} A_\nu$, where each $A_\nu$ is an integrable subset of $F_\nu^n$, and $f_\infty$ is a function on $\A_\infty^n$ satisfying a bound of the form, for any $\gamma \in F^*$ and a constant $C>0$,
\begin{equation} \label{eq:bdcdn}
\int_{\A_\infty^n} f_\infty(x)f_\infty(\gamma x) d\alpha_\infty^n \leq C\alpha_\infty^n(f)\min(1,\min_{\sigma \mid \infty} \|\gamma\|_\sigma^{-1})^n;
\end{equation}
here $\alpha_\infty$ is the Haar measure on $\A_\infty$ compatible with $\alpha_F$ in Theorem \ref{thm:intro_main}. Then
\begin{equation*}
\int_{\mathcal X_n} \left(\sum_{x \in F^n \backslash \{0\}} f(xg) \right)^2 d\mu = (\alpha_F^n(f))^2 + O_F(C\alpha_F^n(f)).
\end{equation*}
\end{theorem}

The reason we require a condition such as \eqref{eq:bdcdn} is that, when applying Theorem \ref{thm:intro_main} to $f$ and $k=2$, we confront an expression of the form
\begin{equation*}
\sum_{u \in \cO_F^*} \int_{\A_\infty^n} f_\infty(x)f_\infty(\gamma u x) d\alpha_\infty^n.
\end{equation*}
Since $\cO_F^*$ is in general an infinite set, in order to show this sum is not too large, we need some bound on the inner integral that approaches zero reasonably quickly as $\max_{\sigma \mid \infty} \|u\|_\sigma$ approaches infinity.

Examples of functions satisfying \eqref{eq:bdcdn} are:
\begin{itemize}
\item any function of the form $f_\infty = \prod_{\sigma \mid \infty} f_\sigma$, where each $f_\sigma: F_\sigma^n \rightarrow \R_{\geq 0}$ is bounded and integrable, with $C = \prod_{\sigma \mid \infty} \sup_{x \in F_\sigma^n} f_\sigma(x)$
\item the characteristic function of a ball or an annulus in $\A_\infty^n$ centered at origin (see Section 2.2 or 5.2 for the metric on $\A_\infty$)
\end{itemize}

For the former, \eqref{eq:bdcdn} follows quickly from the H\"older inequality; for the latter, it is proved in Lemma \ref{lemma:2termest} below. The interest in the functions of the latter kind arises from the idea of Venkatesh \cite{Ven13}, who interpreted rank-$2$ $\Z[\zeta_d]$-modules, where $\zeta_d$ is a primitive $d$-th root of unity, as lattices in $\R^{2\varphi(d)}$ with the rotational symmetry coming from the multiplication by $\zeta_d$, and then exploited this symmetry to finesse an improved lower bound on the sphere packing density in certain dimensions.

The numerous known results that rely on the $k=1,2$ cases of the Rogers integral formula --- e.g. \cite{Sch60}, \cite{AM09}, \cite{AM15}, \cite{GKY??} --- may thus be extended almost immediately to the adelic context, by simply plugging in the estimate of Theorem \ref{thm:intro_second} and repeating the argument verbatim. For instance, we have the following generalization of the famous discrepancy bound of Schmidt \cite[Theorem 1]{Sch60}.

\begin{corollary}
Let $n \geq 3$. Let $\Phi = \{S_V\}_{V > 0}$ be an increasing (ordered by set containment) family of Borel sets $S_V \in (F \otimes \R)^n$ of $\alpha_\infty^n$-measure $V$, whose characteristic functions satisfy \eqref{eq:bdcdn} with uniform $C$. Define, for a rank $n$ torsion-free module $M \subseteq (F\otimes\R)^n$ of covolume $1$,
\begin{equation*}
D(M,V) = \frac{\left|M \cap S_V\right| - V}{V}
\end{equation*}

Choose any non-decreasing function $\psi : \R_{>0} \rightarrow \R_{>0}$ such that $\int_0^\infty \psi(s)^{-1}ds$ converges. Then for almost every $M$ (in the sense of the Haar measure on $\GL(n,F) \backslash G_n / K$),
\begin{equation*}
D(M,V) = O(V^{-1/2}\log V \psi^{1/2}(\log V)).
\end{equation*}
\end{corollary}

\begin{remark}

To obtain the second moment estimate in the case $n=2$, one could either appeal the spectral theory of automorphic forms, as in \cite[Section 4.2]{AM09}, or prove the Rogers formula in case $n=k=2$, as in \cite[Section 8]{Sch60}. The latter, and the cases $n=k$ in general, could in fact be obtained using the method of the present paper, though the fact that $\GL(n,\A_F)$, while transitive on $(\A_F \backslash \{0\})^{n-1}$, is not transitive on $(\A_F \backslash \{0\})^n$ necessitates an additional ingredient in the argument.

\end{remark}

We also prove the following estimate for the rational points on the projective spaces, which may be seen as yet another analogue of the primitive lattice points in the classical context.

\begin{theorem} \label{thm:intro_proj}
Let $n \geq 3$. For $B > 0$, let $P_B(g)$ be the number of the rational points $x \in \mathbb P^{n-1}(F)$ such that the height of $x$ twisted by $g$ is less than or equal to $B$ (see Section 2.4, especially the paragraph below \eqref{eq:htdef}, for the definition of the twisted height). Then there exists a function $f_B:\A_F^n \rightarrow \R$ such that
\begin{equation*}
P_B(g) = \sum_{x \in F^n \backslash \{0\}} f_B(xg).
\end{equation*}

In addition, we have
\begin{equation*}
\int_{\mathcal X_n} (P_B(g))^2 d\mu = (\alpha_F^n(f_B))^2 + O_F(\alpha_F^n(f_B)),
\end{equation*}
and $\alpha_F^n(f_B) = CB^n$ for an explicit constant $C$ depending only on $F$ and $n$.
\end{theorem}
The proof in Section 5.3 provides the explicit formulas for both $f_B$ and $C$.

The higher moments are substantially more challenging to compute, again due to the infinitude of $\cO_F^*$: for the $k$-th moment we encounter a sum of integrals over $(k-1)$-tuples of units. We will visit this topic, along with its applications to lattice problems, in a forthcoming paper.

\subsection{Organization}

Section 2 provides a summary of most of the number-theoretic background, notations and conventions that we use throughout the paper. In Section 3, the ``level 1'' case of Theorem \ref{thm:intro_main}, Theorem \ref{thm:main}, is proved. Theorem \ref{thm:main} is then further extended in Section 4, where we prove the primitive version (Theorem \ref{thm:prim}), the higher level version (Theorem \ref{thm:cong}), and ultimately the main result Theorem \ref{thm:intro_main}. Section 5 is dedicated to the second moments, proving Theorems \ref{thm:intro_second} and \ref{thm:intro_proj} above.

\subsection*{Acknowledgment}

This work was supported by NSF grant CNS-2034176. The author thanks the referee for the numerous helpful comments and suggestions that led to a considerable improvement over the original manuscript.

\section{Preliminaries}

In this section, we clarify the basic facts and conventions that we use throughout this paper. It is by no means sufficient for an introduction or a guide to the subject matter. The reader who seeks such resource may refer to any standard text on algebraic number theory, e.g. Cassels-Fr\"ohlich \cite{CF67}, Weil \cite{Wei74}. For the Tamagawa measure, see Thunder \cite{Thu96} or Weil \cite{Wei82}. For the Hecke operators, see Shimura \cite{Shi71}, Chapter 3; it is written in the classical language, but all the results there extend to the adelic context in a straightforward manner.

\subsection{Number fields}

A number field $F$ is an algebraic field extension over $\Q$ of finite degree $d$. We denote by $\cO_F$ the ring of integers of $F$. Throughout this paper, we fix $F$ once and for all, and work over this $F$ only. For a prime ideal $\gp$ of $F$ (more precisely, of $\cO_F$), we define the norm $N(\gp) = \left|\cO_F : \gp\right|$, the index of $\gp$ as an additive subgroup of $\cO_F$. We extend $N(\cdot)$ to all fractional ideals of $F$ so that it becomes multiplicative on the group of the fractional ideals of $F$. For an element $x \in F$, we define $N(x) = N((x))$.

By a place of $F$, we mean an equivalence class of absolute value defined on $F$, where two absolute values $|\cdot|_1, |\cdot|_2$ are said to be equivalent if and only if $|x|_1 < 1 \Leftrightarrow |x|_2 < 1$. Let us write $\mathcal P_F$ for the set of all places of $F$. Each and every finite place $\nu$, or equivalently, class of nonarchimedean absolute values, corresponds to a prime ideal of $\cO_F$, say $\gp_\nu$, and vice versa. We define $\| \cdot \|_\nu: F \rightarrow \Q$ associated to $\gp_\nu$ to be $\| x \|_\nu = N(\gp_\nu)^{-\ord_\nu x}$, where $\ord_\nu x$ is the exponent of $\gp_\nu$ in the unique factorization of the principal ideal $(x)$. We will write $\f$ for the set of all finite places of $F$.

On the other hand, every infinite place of $F$, i.e. equivalence class of archimedean absolute values, arises from a field embedding $F \rightarrow \C$. There are exactly $d$ different field embeddings $\sigma_1, \ldots, \sigma_d$, $r_1$ of which have their images contained in $\R$, thus called real embeddings, and the rest of which consist of $r_2$ conjugate pairs of complex embeddings, so that $d = r_1 + 2r_2$. Following the common convention, we let $\sigma_1, \ldots, \sigma_{r_1}$ be the real embeddings, and $\sigma_{r_1+j} = \bar\sigma_{r_1+r_2+j}$ for $1 \leq j \leq r_2$ be the pairs of complex embeddings. Each embedding $\sigma$ yields an infinite place on $F$ represented by the archimedean absolute value $x \mapsto |\sigma(x)|$. Note that a conjugate pair of complex embeddings define the same place --- hence there are $r_1 + r_2$ infinite places total. As in the finite case, we refer to each infinite place by the associated embedding or the conjugate pair of embeddings of $F$. We sometimes write $\sigma \mid \infty$ to indicate that $\sigma$ is an infinite place.

Due to the tight connection between places and primes, places are sometimes referred to as primes, including the infinite ones. Also we shall use the notations for places and primes interchangeably when we find it more expedient. For instance, $\|\cdot\|_\gp$ means the absolute value associated to the place $\nu$ corresponding to $\gp$; and the letter $\nu$, while typically used to denote a place, can refer to the prime ideal associated to it.

For each place $\nu$ of $F$, let $F_\nu$ be the completion of $F$ at $\nu$. If in addition $\nu \in \f$, let $\cO_\nu$ be the ring of integers of $F_\nu$, and we also choose its uniformizer $\pi_\nu \in \cO_\nu$, so that $(\pi_\nu)$ is the unique maximal ideal of $\cO_\nu$.

There exists the canonical embedding $\rho$ of $F$ into $\R^{r_1} \times \C^{2r_2}$, defined by
\begin{equation*}
\rho(x) = (\sigma_1(x), \ldots, \sigma_d(x)).
\end{equation*}
The image $\rho(F)$ spans a vector space of dimension $d$ over $\R$, which we may identify with $F \otimes_\Q \R$. We endow an inner product on $F \otimes_\Q \R$ by simply restricting to it the standard inner product on $\R^{r_1} \times \C^{2r_2}$, so that
\begin{equation*}
\langle \rho(x), \rho(y) \rangle = \sum_i \sigma_i(x)\bar\sigma_i(y),
\end{equation*}
and
\begin{equation*}
\|\rho(x)\|^2 = \langle \rho(x), \rho(x) \rangle = \sum_i |\sigma_i(x)|^2.
\end{equation*}
For any $\Z$-basis $\{b_1, \ldots, b_n\}$ of $\cO_F$, the discriminant of $F$ is the quantity
\begin{equation*}
\Delta_F = \left(\det (\sigma_i(b_j))_{1 \leq i,j \leq n} \right)^2.
\end{equation*}
This definition is independent of the choice of the basis. It is known that $|\Delta_F|^{1/2}$ is the covolume of the lattice $\rho(\cO_F)$ in $F \otimes_\Q \R$ with respect to the metric above.


There exists also the logarithm map $\Log : F^* \rightarrow \R^{r_1+r_2}$ defined by
\begin{equation*}
\Log(x) = (\log |\sigma_1(x)|, \ldots, \log |\sigma_{r_1}(x)|, 2\log |\sigma_{r_1+1}(x)|, \ldots, 2\log |\sigma_{r_1+r_2}(x)|).
\end{equation*}
Its kernel is $\mu_F$, the set of the roots of unity in $\cO_F$; we write $|\mu_F| = w_F$. $\Log$ takes $\cO^*_F$ to a lattice in $\R^{r_1+r_2}$ of rank $r:=r_1+r_2-1$, called the unit lattice. Its covolume, with respect to the standard metric on $\R^{r_1+r_2}$, is called the regulator $R_F$.

\subsection{Adeles}

The ring of finite adeles $\A_\f$ of $F$ is the restricted direct product
\begin{equation*}
\A_\f = \prod_{\nu \in \f}\\' F_\nu
\end{equation*}
with respect to $(\cO_\nu)_{\nu \in \f}$, which is the set of all elements $x = (x_\nu)_{\nu \in \f}$ such that all but finitely many $x_\nu \in \cO_\nu$.

We also define the ring of infinite adeles $\A_\infty$ to be
\begin{equation*}
\A_\infty = \prod_{\sigma \mid \infty} F_\sigma.
\end{equation*}
We identify $\A_\infty$ with $F \otimes \R$ by their natural isomorphisms to $\R^{r_1} \times \C^{r_2}$, and assign the metric compatible with that of $F \otimes \R$ under this identification.

The adele ring $\A_F$ of $F$ is the restricted direct product of $F_\nu$ over all places of $F$, namely
\begin{equation*}
\A_F = {\prod_{\nu \in \mathcal P_F}}' F_\nu = \A_\f \times \A_\infty.
\end{equation*}
For $x$ an element of $\A_F$ (or, more generally, of $G(\A_F)$ for an algebraic group $G$, such as $\A_F^n$ or $\GL(n,\A_F)$), let us write $x_\nu$ for its coordinate at place $\nu$. Similarly, for $\mathcal P \subseteq \mathcal P_F$, let us write $x_{\mathcal P} = \prod_{\nu \in \mathcal P}' x_{\nu}$. There is the natural diagonal embedding $F \hookrightarrow \A_F$ that is the product of the embeddings $F \hookrightarrow F_\nu$ over all $\nu \in \mathcal P_F$. We identify $F$ with the image of this embedding.

$\A_F$ is equipped with the restricted product topology, whose base consists of all sets of the form $\prod_{\nu \in \mathcal P_F} O_\nu$, where each $O_\nu \subseteq F_\nu$ is open, and for all but finitely many finite places $O_\nu = \cO_\nu$. With this topology, $\A_F$ is locally compact, separable, and regular Hausdorff, and thus much of the well-known results in functional analysis apply. 
In addition, $F \subseteq \A_F$ becomes a discrete and cocompact (i.e. $\A_F/F$ is compact) subgroup under this topology.

On each $F_\nu$ we assign an ``almost-norm'' $\| \cdot \|_\nu$ as follows. If $\nu$ is a finite prime lying over a rational prime $p$, then $\|x_\nu\|_\nu$ is the absolute value associated to $\nu$, as defined in the previous section. If $\nu$ is real, then $\|x_\nu\|_\nu$ is the standard absolute value on $\R$, and if $\nu$ is complex, then $\|x_\nu\|_\nu$ is the square of the standard absolute value on $\C$. For $x \in \A_F$, we let
\begin{equation*}
\| x \|_{\A_F} = \prod_{\nu} \|x_\nu\|_\nu.
\end{equation*}
The product formula states that if $x \in F^*$, then $\|x\|_{\A_F} = 1$.

Another important fact we will use frequently is the strong approximation (\cite{Kne65}; see also \cite[Theorem 2.3]{Rap13}), which states that, for any finite set of places $S \subseteq \mathcal P_F$, and a connected absolutely almost simple algebraic group $G$ over $F$ such that $G$ is simply connected and $G_S = \prod_{\nu \in S} G(F_\nu)$ is noncompact, $G(F)$ is dense in $G(\A^S) = \prod_{\nu \not\in S}' G(F_\nu)$. In particular, this applies for $G = \mathbf A^n$ (the affine $n$-space) and $G = \SL_n$, but not for $G = \GL_n$.



\subsection{Tamagawa measure}

For each place $\nu$ of $F$, we define the measure $\alpha_\nu$ on $F_\nu$ as follows:
\begin{itemize}
\item If $\nu$ is finite, $\alpha_\nu$ is the Haar measure on $F_\nu$ normalized so that $\alpha_\nu(\mathcal{O}_\nu) = 1$.
\item If $\nu$ is real, $\alpha_\nu$ is the usual Lebesgue measure on $\R$.
\item If $\nu$ is complex, $\alpha_\nu$ is twice the usual Lebesgue measure on $\C$.
\end{itemize}
We also write
\begin{equation*}
\alpha_\f = \prod_{\nu \in \f} \alpha_\nu, \alpha_\infty = \prod_{\sigma \mid \infty} \alpha_\sigma,
\end{equation*}
for the corresponding measures on $\A_\f$ and $\A_\infty$ respectively, and let
\begin{equation*}
\alpha_F = |\Delta_F|^{-\frac{1}{2}}\alpha_\f \alpha_\infty.
\end{equation*}
This is called the {Tamagawa measure} on $\A_F$. The choice of the constant factor ensures that $\alpha_F(\A_F/F) = 1$.

On each $\GL(n, F_\nu)$, there is the invariant measure
\begin{equation*}
\omega_{n,\nu}(A) = |\det A|_\nu^{-n} \prod_{1 \leq i, j \leq n} \alpha_\nu(a_{ij}),
\end{equation*}
where $a_{ij}$ refers to the $(i,j)$-entry of $A$. On $\GL(n, \A_F)$, we have the {Tamagawa measure}
\begin{equation*}
\omega_n = \prod_{\nu \mid \infty} \omega_{n,\nu} \times \prod_{\nu \in \f} (1-\|\pi_\nu\|_\nu)^{-1} \omega_{n, \nu}.
\end{equation*}

Define
\begin{equation*}
G_n = \{A \in \GL(n, \A_F) : \|\det A\|_{\A_F} = 1 \}, \Gamma_n = \GL(n, F).
\end{equation*}
$G_n$ is equipped with a measure $\mu_n$ such that
\begin{equation*}
\omega_n = \mu_n \times \beta,
\end{equation*}
where $\beta$ is the Haar measure on $\R_{>0}$ given by $d\beta(x) = dx/x$. Moreover, $\mu_n(\Gamma_n \backslash G_n)$ is finite (\cite[Section 3]{Thu96}). We write ${\mathcal X}_n = \Gamma_n \backslash G_n$ for short, and normalize $\mu_n$ so that $\mu_n({\mathcal X}_n) = 1$.

\subsection{Height}

We follow Thunder \cite[Part I, Section 1]{Thu93} for the notion of height; see also Schmidt \cite[Section 1.1]{Sch67}. On each place $\nu$, the height of $X_\nu \in \mathrm{Mat}_{m \times n}(F_\nu)$ is defined as follows. First suppose that $m =1$, i.e. $X_\nu = (a_1, \ldots, a_n) \in F_\nu^n$. Then we let
\begin{equation*}
H_\nu(X_\nu) = \begin{cases} \max_i \| a_i \|_\nu & \mbox{if $\nu$ is finite} \\ \sqrt{a_1^2 + \ldots + a_n^2} & \mbox{if $\nu$ is real} \\ |a_1|^2 + \ldots + |a_n|^2 & \mbox{if $\nu$ is complex.} \end{cases}
\end{equation*}
This extends naturally to general $m \leq n$, where we write $x_i$ for the $i$-th row of $X_\nu$, and define
\begin{equation*}
H_\nu(X_\nu) = H_\nu(x_1 \wedge \cdots \wedge x_m).
\end{equation*}
Here recall that $x_1 \wedge \cdots \wedge x_m \in F_\nu^{\binom{n}{m}}$.

For $\nu$ infinite, it is sometimes helpful to note the following equivalent definition of $H_\nu$. For a (not necessarily square) matrix $X$ of real or complex entries, let us write $\left|\det\right| X = \sqrt{\det X\bar X^\mathrm{tr}}$. Then we have
\begin{equation*}
H_\nu(X_\nu) = \| \left|\det\right| X_\nu \|_\nu.
\end{equation*}

For $X \in \mathrm{Mat}_{m \times n}(\A_F)$ and $X_\infty \in \mathrm{Mat}_{m \times n}(\A_\infty)$, we define respectively
\begin{equation} \label{eq:htdef}
H(X) = \prod_{\nu \in \mathcal P_F} H_\nu(X_\nu),\, H_\infty(X_\infty) = \prod_{\sigma \mid \infty} H_\sigma(X_\sigma).
\end{equation}

The product formula implies that $H(X)$ is invariant under the multiplication by $F^*$. This property allows us to define the height over the projective space $\mathrm{Gr}_{n,m}(F)$ using $H$. Let us recall the definition, although we only need the $m=1$ case in this paper, for which $\mathrm{Gr}_{n,m}(F) = \mathbb P^{n-1}(F)$. For $L \in \mathrm{Gr}_{n,m}(F)$, and $X \in \mathrm{Mat}_{m \times n}(F)$ any choice of a representative of $L$, i.e. a matrix whose row vectors span $L$ over $F$, we define the height of $L$ to be $H(X)$. More generally, for $g \in \GL(n, \A_F)$, following Thunder \cite{Thu93} we define the height of $L$ twisted by $g$ to be $H(Xg)$.

Later we will need a few simple facts about $H_\infty$; we state and prove them below.

\begin{lemma} \label{lemma:height}
Let $X$ be an $m \times n$ matrix with entries in $\A_\infty$. Denote by $x_i$ the $i$-th row vector of $X$. Then
\begin{equation*}
H_\infty(X) \leq \prod_{i=1}^m H_\infty(x_i).
\end{equation*}
Also, for $n' < n$, if $Y$ is an $m \times n'$ matrix consisting of any choice of the $n'$ columns of $X$, then
\begin{equation*}
H_\infty(Y) \leq H_\infty(X).
\end{equation*}
\end{lemma}

\begin{proof}
It suffices to prove the corresponding inequalities in case $X$ is a complex-valued matrix, namely
\begin{equation*}
\left|\det\right| X \leq \prod_{i=1}^m \left|\det\right| x_i,\mbox{and } \left|\det\right| Y \leq \left|\det\right| X.
\end{equation*}
The former follows from the well-known inequality
\begin{equation*}
\left|\det\right| X \leq \prod_{i=1}^m \|x_i\|
\end{equation*}
which can be seen by observing that $\left|\det\right| X = \|x_1 \wedge \ldots \wedge x_m\|$, which is maximized for fixed $\|x_1\|, \ldots, \|x_m\|$ when the $x_i$'s are pairwise orthogonal. The latter is an immediate consequence of the Cauchy-Binet formula.
\end{proof}

Given a fractional ideal $I$ of $F$, we can associate with it the lattice $\rho(I)$ in the Euclidean space $F \otimes_\Q \R$. Denote by $\Delta_F$ the discriminant of $F$. It is known that (\cite[Theorem 1]{Sch67}, also see \cite[Theorem 1]{Thu93})
\begin{equation} \label{eq:det}
\det \rho(I) = |\Delta_F|^{\frac{1}{2}}N_F(I).
\end{equation}
In general, given an $\cO_F$-submodule $M$ of $F^n$, we can define $\rho(M) \subseteq (F \otimes_\Q \R)^n$ by naturally extending $\rho$ to $F^n$.

\begin{lemma} \label{lemma:ideal_min}
For a lattice $L$ in a Euclidean space, let $\lambda_1(L)$ be the length of a shortest nonzero vector of $L$. Then for any fractional ideal $I$ of $F$
\begin{equation*}
\sqrt{d}N(I)^{\frac{1}{d}} \leq \lambda_1(\rho(I)).
\end{equation*}

Similarly, take nonzero $x \in F^n$, and consider the rank one $\cO_F$-submodule $I \cdot x \subseteq F^n$. Then
\begin{equation*}
\sqrt{d}N(I)^{\frac{1}{d}}H_\infty(x)^{\frac{1}{d}} \leq \lambda_1(\rho(I \cdot x)).
\end{equation*}

\end{lemma}
This lemma is well-known, and there is also a ``reverse inequality''
\begin{equation*}
\lambda_1(\rho(I)) \leq \sqrt{d\Delta_F^\frac{1}{d}}N(I)^{\frac{1}{d}};
\end{equation*}
see e.g. \cite[Lemma 2.9]{LPR10}. 
\begin{proof}
Take any nonzero $a \in I$. Then
\begin{equation*}
\|\rho(a)\|^2 = \sum_i |\sigma_i(a)|^2 \geq d\left(\prod_i |\sigma_i(a)|^2\right)^{\frac{1}{d}} \geq dN(I)^\frac{2}{d}.
\end{equation*}
The middle inequality is the arithmetic-geometric mean inequality; the one on the right follows from the fact that $J \subseteq I \Rightarrow N(J) \geq N(I)$ for fractional ideals $I,J$. The second inequality follows similarly.
\end{proof}

The following is in a sense a generalization of \eqref{eq:det}.

\begin{lemma} \label{lemma:index}
Let $S \subseteq \A_\f^n$ be an integrable set. Let $M$ be the set of all elements $x \in F^n$ such that $x_\f \in S$. Then ratio of the natural density of $\rho(M)$, that is,
\begin{equation*}
\lim_{V \rightarrow \infty} \frac{\mathrm{vol}(\rho(M) \cap B_V)}{\mathrm{vol}(B_V)}
\end{equation*}
where $B_V \subseteq (F \otimes \R)^n$ is the ball of volume $V$, to that of $\rho(\cO_F^n)$ is given by $\alpha_\f^n(S)$.
\end{lemma}

As a corollary, if $\rho(F^n \cap S)$ forms a lattice in $(F \otimes \R)^n$, then its determinant(covolume) is given by $|\Delta_F|^{n/2}(\alpha_\f^n(S))^{-1}$.

\begin{proof}
Choose an ideal $I = \prod_{\nu \in \f} \gp_\nu^{k_\nu} \subseteq \cO_F$, and let $S_0 = \prod_{\nu \in \f} \pi_\nu^{k_\nu}\cO_\nu$ be the associated neighborhood of the identity of $\A_\f$. For the lemma, it suffices to consider the case $S = v + S_0^n$ for some $v \in (F \cap \A_\f)^n$, since $S$ in general can be approximated arbitrarily well by taking unions and/or complements of finitely many sets of such form. But in this case, the lemma is obvious because $|\cO_F : I|^{-n} = \alpha_\f^n(S_0) = \alpha_\f^n(S)$.
\end{proof}

\subsection{Hecke operators}

This paper employs two distinct families of the Hecke operators, each for different purposes. To describe the first family, let $\phi: \mathcal X_n \rightarrow \C$ be a measurable function. Choose a prime $\gp$ of $F$, and write $a_\gp = \diag(\pi_p,1, \ldots, 1), K_\gp = \GL(n,\cO_\gp)$. Also for $r \in \R$ write
\begin{equation*}
r_\infty = (\underbrace{1, \ldots, 1}_{\mbox{\tiny finite places}}, \underbrace{r, \ldots, r}_{\mbox{\tiny infinite places}}) \in \A_F.
\end{equation*}
The Hecke operator $T_\gp$ is defined as the integral
\begin{equation*}
T_\gp \phi (g) = \frac{1}{\omega_\gp(K_\gp a_\gp K_\gp)} \int_{K_\gp a_\gp K_\gp} \phi(N\gp_\infty^\frac{1}{nd}gh) d\omega_\gp(h).
\end{equation*}
Here, inside the argument of $\phi$, $h$ should be understood as an element of $\GL(n,\A_F)$ whose projection to $\GL(n,F_\gp)$ is $h$ and the rest are the identity. It is clear that $T_\gp \phi$ is also a measurable function on $\mathcal X_n$. 

It is sometimes convenient to realize $T_\gp$ as a sum rather than an integral. To this end, let us choose a set of coset representatives $h \in \GL(n,F_\gp)$ of $K_\gp a_\gp K_\gp$, so that
\begin{equation*}
K_\gp a_\gp K_\gp = \amalg_h h K_\gp.
\end{equation*}
Let $\mathcal R_\gp \subseteq \cO_\gp$ be a set of the coset representatives of $\cO_\gp  / \gp\cO_\gp$. By the theory of the Hecke operators (see e.g. \cite[Chapter 3]{Shi71}; Shimura develops the theory over $\Z$, but his argument applies to any PID, the main tool being the theory of the Smith normal form), we can choose the set of the representatives to be
\begin{equation*}
\mathcal L =
\left\{
h(j; a_{1,\gp}, \ldots, a_{n-j-1,\gp}) : {0 \leq j \leq n-1, a_{i,\gp} \in \mathcal R_\gp} \right\}.
\end{equation*}
where
\begin{equation*}
h(j; a_{1,\gp}, \ldots, a_{n-j-1,\gp}) =
\begin{pmatrix}
\mathrm{Id}_j &              &          &           & \\
                       & \pi_\gp  &  a_{1,\gp}  & \ldots & a_{n-j-1, \gp}            \\
                       &              & 1       &           &  \\
                       &              &          & \ddots & \\
                       &              &          &            & 1
\end{pmatrix}
\end{equation*}
The cardinality of $\mathcal L$ is $1 + N\gp + \ldots + N\gp^{n-1}$, which is also equal to $\omega_\gp(K_\gp a_\gp K_\gp)$. Now one may write
\begin{equation*}
T_\gp \phi (g) = \frac{1}{\omega_\gp(K_\gp a_\gp K_\gp)} \sum_{h \in \mathcal L} \phi(N\gp_\infty^\frac{1}{nd}gh).
\end{equation*}

The second family of the Hecke operators is denoted by the letter $\mathcal T$, and is more of a combinatorial device than an operator on the space of automorphic forms. This time, choose $m < n$, and write $K_\nu = \GL(m,\F_\nu)$ for $\nu \in \f$.
For a sequence $\nu^{a_1} \supseteq \ldots \supseteq \nu^{a_m}$ of nonzero ideals in $\cO_\nu$ and a measurable function $f$ on $\mathrm{Mat}_{m \times n}(\A_F)$ invariant under $K_\nu$ from the left, define the operator $\mathcal T(\nu^{a_1}, \ldots, \nu^{a_m})$ by
\begin{equation*}
\mathcal T(\nu^{a_1}, \ldots, \nu^{a_m})f(X) = \int_{K_\nu\,\mathrm{diag}(\pi_\nu^{a_1}, \ldots, \pi_\nu^{a_m}) K_\nu} f(\gamma^{-1} X) d\omega_\nu(\gamma).
\end{equation*}
Clearly the output is also a measurable function on $K_\nu \backslash \mathrm{Mat}_{m \times n}(\A_F)$.
There is also the invariant called degree, defined by
\begin{equation*}
\deg \mathcal T(\nu^{a_1}, \ldots, \nu^{a_m}) = \mbox{($\#$ of cosets of $K_\nu$ in $K_\nu\,\mathrm{diag}(\pi_\nu^{a_1}, \ldots, \pi_\nu^{a_m}) K_\nu$)}.
\end{equation*}
The operators of the form $\mathcal T(\nu^{a_1}, \ldots, \nu^{a_m})$ together with $\R$ (which acts by the scalar multiplication), under the usual addition and the composition operations, generate a commutative ring with $1$ called the Hecke ring. The map $\deg$ extends to a homomorphism from the Hecke ring to $\Z$.

In general, for a sequence $I_1 \supseteq \ldots \supseteq I_m$ of nonzero ideals in $\cO_F$, define
\begin{equation*}
\mathcal T(I_1, \ldots, I_m)f(X) = \left(\prod_{\nu \in \f} \mathcal T(\nu^{\mathrm{ord}_\nu(I_1)}, \ldots, \nu^{\mathrm{ord}_\nu(I_m)})  \right)f(X),
\end{equation*}
and also, for an ideal $I \neq 0$ of $\cO_F$, define
\begin{equation*}
\mathcal T(I)f(X) = \sum_{I_1 \supseteq \ldots \supseteq I_m \atop I_1 \cdots I_m = I} \mathcal T(I_1, \ldots, I_m)f(X).
\end{equation*}
Both are elements of the Hecke ring.

\subsection{Miscellaneous}
Throughout this paper, we adopt the following notation. We identify an element $(x_1, \ldots, x_k) \in (\A_F^n)^k$ with a $k \times n$ matrix
\begin{equation*}
X :=
\begin{pmatrix}
x_{11} & \ldots & x_{1n} \\
  & \vdots  &  \\
x_{k1} & \ldots & x_{kn} 
\end{pmatrix} \in \mathrm{Mat}_{k \times n}(\A_F)
\end{equation*}
whose $i$-th row is equal to $x_i$. When we say $X$ is linearly independent or etc., we mean the row vectors of $X$ has those properties. In particular, for $f: (\A_F^n)^k \rightarrow \R$ and $g \in \GL(n,\A_F)$, we write $f(Xg)$ for $f(x_1g, \ldots, x_kg)$.

We omit the subscripts when there exists no ambiguity as to what they should be, e.g. $d\mu = d\mu_n, d\omega_\nu = d\omega_{n,\nu}$, and so on. Also, if $f$ is a function defined on $G(\A_F)$ for some algebraic group $G$ over $F$ and $x_S \in G_S$ for some finite set $S \subseteq \mathcal P_F$, $f(x_S)$ is to be understood as $f$ evaluated at $x_\nu$ for $\nu \in S$ and at $\mathrm{Id}_{G(F_\nu)}$ for $\nu \not \in S$. We extend this convention to other similar situations, e.g. if $x \in G(F)$ then $f(x)$ is $f$ evaluated at the diagonal embedding of $x$ to $G(\A_F)$.

\section{The ``level 1'' case}

\subsection{A reduction, and a word about the proof}

Let $\mathcal{S}$ be the set of all functions $f: \A_F^n \rightarrow \R$ such that $f = f_\infty  f_\f = f_\infty \cdot \prod_{\nu \in \f}f_\nu$, where $f_\infty$ is a Riemann integrable function on $(F \otimes \R)^n$ that is bounded and compactly supported, and for $\nu$ finite $f_\nu$ is the characteristic function of $\pi_\nu^{e_\nu} \cO_\nu^n$ for some $e_\nu \in \Z$, all but finitely many of which are $0$. Accordingly, let $\mathcal{S}^k$ be the set of all functions $f: (\A_F^n)^k \rightarrow \R$ such that $f(x_1, \ldots, x_k) = \prod_{i=1}^k f^{(i)}(x_i)$ for $f^{(i)} \in \mathcal{S}$. Thus $f^{(i)}_\nu$ is the characteristic function of $\prod_{\nu \in \f} \pi_\nu^{e_\nu^{(i)}} \cO_\nu^n$ for some $e_\nu^{(i)} \in \Z$, all but finitely many of which are $0$. Write $I^{(i)} = \prod_{\nu \in \f} \gp_\nu^{e_\nu^{(i)}}$ for the ideal corresponding to the support of $f_\f^{(i)}$.

The goal of this section is to prove

\begin{theorem} \label{thm:main}
Let $1 \leq m \leq k < n$ be integers. Let $D$ be an $m \times k$ row-reduced echelon form over $F$ of rank $m$, and $f \in \mathcal S^k$. Then
\begin{equation} \label{eq:main}
\int_{\mathcal X_n} \sum_{X \in (F^n)^m \atop \mathrm{indep.}} f(D^\tr Xg) d\mu_n(g) = \int_{(\A_F^n)^m} f(D^\tr X) d\alpha_F^{nm}(X).
\end{equation}
\end{theorem}

One might reasonably name this the level $1$ subcase of Theorem \ref{thm:intro_main}, since for $f \in \mathcal S^k$ the sum inside the left-hand side of \eqref{eq:main} is invariant under the right action of $\GL(n,\cO_\nu)$ for all $\nu \in \f$.

A few remarks are in order before we proceed:
\begin{enumerate}[(i)]
\item Without loss of generality, by reordering $f^{(i)}$'s and the columns of $D$ if necessary, we may assume that the pivots of $D$ are its first $m$ columns, i.e. $D$ is of the form
\begin{equation*}
D = \begin{pmatrix}
1 &            &     &  *  & \cdots & * \\
   & \ddots  &    & \vdots &   & \vdots \\
   &             & 1 &  *  & \cdots & *
\end{pmatrix}.
\end{equation*}

\item There is one small but important trick that seems to facilitate the computations to come later. For $g \in G_n$ denote by $g^*$ its inverse transpose. The map $g \mapsto g^*$ then induces a degree two automorphism on ${\mathcal X}_n$, since it is compatible with the action of $\Gamma_n$. In particular, it holds that
\begin{equation*}
\int_{{\mathcal X}_n} \sum_{X \in (F^n)^m \atop \mathrm{indep.}} f(D^\tr Xg) d\mu_n(g) = \int_{{\mathcal X}_n} \sum_{X \in (F^n)^m \atop \mathrm{indep.}} f(D^\tr Xg^*) d\mu_n(g),
\end{equation*}
and in the proof of Theorem \ref{thm:main} to follow, we work with the right-hand side.

\end{enumerate}

The bulk of our effort for proving Theorem \ref{thm:main} is devoted to the following asymptotic, which one may take as a discrete analogue of \cite[Theorem 2]{Rog55}.

\begin{proposition} \label{prop:main}
There exists a sequence $\{\gp_i\}_{i \in \Z_{>0}}$ of primes of $F$ with $\lim_{i \rightarrow \infty} N\gp_i = \infty$ such that, for each $g \in G_n$, 
\begin{equation} \label{eq:inprop}
T_{\gp_i}\left(\sum_{X \in (F^n)^m \atop \mathrm{indep.}} f(D^\tr Xg^*)\right) \rightarrow \int_{(\A_F^n)^m} f(D^\tr X) d\alpha_F^{nm}(X)
\end{equation}
as $i \rightarrow \infty$.
\end{proposition}

The proof of this statement we provide later in this section may appear technical due to the use of the adelic language, but the main idea is in fact quite simple: the left-hand side of \eqref{eq:inprop}, when unraveled, reduces to a ``lattice-point counting estimate.'' We illustrate the point by briefly explaining the simplest case, where $F = \Q, k=m=1, D=(1), g = \mathrm{Id}$, in the classical language. For $\phi : X_n \rightarrow \C$ and a rational prime $p$, $T_p\phi(\mathrm{Id})$ is the average of $\phi$ over the sublattices of $\Z^n$ of index $p$ scaled by $p^{-1/n}$. Denote those sublattices of $\Z^n$ by $\Lambda_1, \ldots, \Lambda_{p+1}$. Then, for $\phi(g) = \sum_{x \in \Z^n \atop x \neq 0} f(xg)$, where $f:\R^n \rightarrow \R$ is, say, the characteristic function of a bounded open set $B$, $T_p\phi(\mathrm{Id})$ equals
\begin{equation*}
\frac{1}{p+1}\sum_{i=1}^{p+1} \sum_{x \in \Lambda_i \atop x \neq 0} f(p^{-1/n}x).
\end{equation*}
But upon inspecting the distribution of the points of $\Lambda_1, \ldots, \Lambda_{p+1}$ --- they are ``equidistributed mod $p$'' in $\Z^n$, other than at the origin --- one observes that this equals
\begin{equation*}
\frac{1}{p+1}\sum_{x \in \Z^n} f(p^{-1/n}x) + \frac{p}{p+1} \sum_{x \in p\Z^n \atop x \neq 0} f(p^{-1/n}x).
\end{equation*}
As $p \rightarrow \infty$, the first sum converges to $\vol\,B = \int f dx$, because $B$ has $\approx p \cdot \vol\,B$ points of $p^{-1/n}\Z^n$; the second sum vanishes (but, if we allow $x = 0$, converges to $f(0)$). Our proof of the general case has more or less the same structure as this toy example; although the setup is a bit more intricate and the error estimates are longer, the ``equidistribution mod $\gp$,'' Lemma \ref{lemma:crux} below, lies at the heart of the argument, and everything else is built around it.

The two lemmas below, combined with Proposition \ref{prop:main}, yield Theorem \ref{thm:main}. As with Proposition \ref{prop:main}, both of them have counterparts in the original argument by Rogers \cite{Rog55}.

\begin{lemma} \label{lemma:hecke_invt}
Let $\phi: {\mathcal X}_n \rightarrow \R$ be integrable. Then
\begin{equation*}
\int_{{\mathcal X}_n} \phi d\mu = \int_{{\mathcal X}_n} T_\gp \phi d\mu.
\end{equation*}
\end{lemma}
\begin{proof}
This is immediate from the right invariance of $d\mu$.
\end{proof}

This lemma is, in a sense, a ``correction'' of the error in \cite{Rog55} that we pointed out in the introduction.
The following lemma is an adaptation of \cite[Theorem 1]{Rog55} to our context.

\begin{lemma} \label{lemma:equiv}
Let $\{\gp_i\}_{i \in \Z_{>0}}$ be a sequence of primes of $F$ such that $\lim_{i \rightarrow \infty} N\gp_i = \infty$. Let $\phi: {\mathcal X}_n \rightarrow \R$ be a measurable function such that, for almost every $g \in {\mathcal X}_n$, $T_{\gp_i} \phi(g)$ converges to a finite real number $I$ as $i \rightarrow \infty$. Then $\phi$ is integrable, and
\begin{equation*}
\int_{{\mathcal X}_n} \phi d\mu = I.
\end{equation*}
\end{lemma}
\begin{proof}

For any real-valued function $F$ and $h \in \R$, write $[F]_h := \min(F, h)$. For any $h > I$, the dominated convergence theorem implies
\begin{equation*}
\int_{\mathcal X_n} \left[T_{\gp_i}\phi\right]_h d\mu \rightarrow I
\end{equation*}
as $i \rightarrow \infty$. Also by Lemma \ref{lemma:hecke_invt}, we have
\begin{equation*}
\int_{\mathcal X_n} [\phi]_h d\mu = \int_{\mathcal X_n} T_{\gp_i}[\phi]_h d\mu \leq \int_{\mathcal X_n} [T_{\gp_i}\phi]_h d\mu.
\end{equation*}
Taking $i \rightarrow \infty$ and then $h \rightarrow \infty$ here, by the monotone convergence theorem we obtain the upper bound
\begin{equation*}
\int_{\mathcal X_n} \phi d\mu \leq I.
\end{equation*}
In particular, this shows that $\phi$ is integrable.

On the other hand, Fatou's lemma and Lemma \ref{lemma:hecke_invt} imply that
\begin{equation*}
I = \int_{\mathcal X_n} \lim_{i \rightarrow \infty} T_{\gp_i}\phi d\mu \leq \lim_{i \rightarrow \infty} \int_{\mathcal X_n} T_{\gp_i}\phi d\mu = \int_{\mathcal X_n} \phi d\mu,
\end{equation*}
which is the lower bound that we need.
\end{proof}

\begin{proof}[Proof of Theorem \ref{thm:main}, assuming Proposition \ref{prop:main}]

By Proposition \ref{prop:main},
\begin{equation*}
\phi(g) = \sum_{X \in (F^n)^m \atop \mathrm{indep.}} f(D^\tr Xg^*)
\end{equation*}
satisfies the assumptions of Lemma \ref{lemma:equiv}. Thus
\begin{equation*}
\int_{\mathcal X_n} \phi d\mu = \int_{(\A_F^n)^m} f(D^\tr X) d\alpha_F^{nm}(X).
\end{equation*}
But the left-hand side here is equal to the left-hand side of \eqref{eq:main}, as remarked in comment (ii) under the statement of Theorem \ref{thm:main}.
\end{proof}

\subsection{Proof of Proposition \ref{prop:main}: preparation}

First, let us determine the sequence $\{\gp_i\}_{i \in \Z_{> 0}}$. We choose it to be the sequence of the primes $\gp$, in the ascending order of the norm, satisfying
\begin{enumerate}[(i)]
\item $\gp = (\tilde{\pi}_\gp)$ is principal. If necessary, we adjust $\pi_\gp$ by a factor of $\cO_\gp^*$, so that $\tilde{\pi}_\gp$ maps to $\pi_\gp$ under the natural embedding $\cO_F \hookrightarrow \cO_\gp$.
\item $\gp$ is coprime to $I^{(i)}$ for any $i$.
\end{enumerate}
By the Tchebotarev density theorem, there are infinitely many such $\gp$. 

The next, lengthier, step is to write out the left-hand side of \eqref{eq:inprop} to recast our problem as that of ``lattice-point counting,'' loosely speaking. We first need some preparations.

Fix $g \in G_n$ from now on. Define $\eta \in \A_F^*$ by $\eta_\nu = \det g_\nu^*$; we have $\|\eta\|_{\A_F} = 1$. Also write $a_\eta = \diag(\eta,1,\ldots,1)$. Then $g^*a_\eta^{-1}$ is an element of $\SL(n,\A_F)$, where the strong approximation applies. Thus there exists $k \in \SL(n, F)$ such that
\begin{equation} \label{eq:s-a}
(kg^*)_\f = l \cdot (a_\eta)_\f
\end{equation}
for some $l \in \prod_{\nu \in \f} \SL(n, \cO_\nu)$ arbitrarily close to the identity. Let us make $l$ so close that, for each $\nu \in \f$, $l_\nu$ is equal to the identity matrix modulo a sufficiently large power of $\pi_\nu$ (depending only on $g$ and $f$), so that $l_\nu^{-1}$ fixes the set $\left(\prod_i \pi_\nu^{e_\nu^{(i)}}\cO_\nu^n\right) \cdot (a_\eta)_\nu^{-1}$ when multiplied from the right.

In what follows, $\gp$ is a prime that satisfies the two conditions (i),(ii) above, and has a sufficiently large norm, so that it satisfies the additional condition
\begin{enumerate}
\item[(iii)] $g_\gp \in K_\gp$. This implies $\eta_\gp \in \cO_\gp^*$, so that $(a_\eta)_\gp \in K_\gp$.
\end{enumerate}

Our starting point is the following lemma, that simplifies the impact of $g$ on our subsequent computations.

\begin{lemma}
With all the notations and assumptions above, we have
\begin{equation*}
T_\gp\left(\sum_{X \in (F^n)^m \atop \mathrm{indep.}} f(D^\tr Xg^*)\right) = \frac{1}{\omega_\gp(K_\gp a_\gp K_\gp)} \int_{K_\gp a_\gp K_\gp} \sum_{X \in (F^n)^m \atop \mathrm{indep.}} \tilde f(N\gp_{\infty}^{-\frac{1}{nd}} D^\mathrm{tr}X h^*a_\eta) d\omega_\gp(h)
\end{equation*}
for some $\tilde f \in \mathcal S^k$ such that $f_\f = \tilde f_\f$ and $\alpha_\infty^{mn}(f_\infty \circ D^\tr) = \alpha_\infty^{mn}(\tilde f_\infty \circ D^\tr)$.
\end{lemma}
\begin{proof}
From the definition of $T_\gp$, we have
\begin{equation*}
T_\gp\left(\sum_{X \in (F^n)^m \atop \mathrm{indep.}} f(D^\tr Xg^*)\right) = \frac{1}{\omega_\gp(K_\gp a_\gp K_\gp)} \int_{K_\gp a_\gp K_\gp} \sum_{X \in (F^n)^m \atop \mathrm{indep.}} f(D^\tr X(N\gp_{\infty}^\frac{1}{nd}gh)^*) d\omega_\gp(h).
\end{equation*}
The set of independent $m$-tuples of $F^n$ is invariant under the right multiplication by $\GL(n,F)$, so we can replace $X$ by $Xk$ above. Also, by \eqref{eq:s-a},
\begin{equation*}
k(N\gp_{\infty}^\frac{1}{nd}gh)^* = N\gp_{\infty}^{-\frac{1}{nd}}kg^*h^* = N\gp_{\infty}^{-\frac{1}{nd}}(kg^*)_\infty l(a_\eta)_\f h^*.
\end{equation*}
Plugging this in, we obtain
\begin{equation*}
\frac{1}{\omega_\gp(K_\gp a_\gp K_\gp)} \int_{K_\gp a_\gp K_\gp} \sum_{X \in (F^n)^m \atop \mathrm{indep.}} f(N\gp_{\infty}^{-\frac{1}{nd}} D^\mathrm{tr}X(kg^*)_\infty l (a_\eta)_\f \,h^*) d\omega_\gp(h).
\end{equation*}
From our assumptions on $l$ and $a_\eta$, we see that this is equal to
\begin{align*}
&\frac{1}{\omega_\gp(K_\gp a_\gp K_\gp)} \int_{K_\gp a_\gp K_\gp} \sum_{X \in (F^n)^m \atop \mathrm{indep.}} f(N\gp_{\infty}^{-\frac{1}{nd}} D^\mathrm{tr}X(kg^*)_\infty h^*l(a_\eta)_\f ) d\omega_\gp(h) \\
&=\frac{1}{\omega_\gp(K_\gp a_\gp K_\gp)} \int_{K_\gp a_\gp K_\gp} \sum_{X \in (F^n)^m \atop \mathrm{indep.}} f(N\gp_{\infty}^{-\frac{1}{nd}} D^\mathrm{tr}X(kg^*)_\infty h^*(a_\eta)_\f ) d\omega_\gp(h) \\
&=\frac{1}{\omega_\gp(K_\gp a_\gp K_\gp)} \int_{K_\gp a_\gp K_\gp} \sum_{X \in (F^n)^m \atop \mathrm{indep.}} f(N\gp_{\infty}^{-\frac{1}{nd}} D^\mathrm{tr}Xh^*a_\eta ((kg^*)^{-1}a_\eta)_\infty^{-1}) d\omega_\gp(h).
\end{align*}
To simplify, we introduce $\tilde f(X) := f_\f(X) f_\infty(X((kg^*)^{-1}a_\eta)_\infty^{-1})$. Since $\det ((kg^*)^{-1}a_\eta)_\sigma = 1$ for all $\sigma \mid \infty$, $f_\infty$ and $\tilde f_\infty$ have the same volume, and by the same principle, so do $f_\infty \circ D^\tr$ and $\tilde f_\infty \circ D^\tr$. Therefore, the above is equal to
\begin{equation*}
\frac{1}{\omega_\gp(K_\gp a_\gp K_\gp)} \int_{K_\gp a_\gp K_\gp} \sum_{X \in (F^n)^m \atop \mathrm{indep.}} \tilde f(N\gp_{\infty}^{-\frac{1}{nd}} D^\mathrm{tr}X h^*a_\eta) d\omega_\gp(h),
\end{equation*}
as desired.
\end{proof}

In order not to overburden ourselves with notations, in what follows below, we abuse the language slightly and continue writing $f$ in place of $\tilde f$. It does not make a difference in the end, since we are interested in volume computations, and $f$ and $\tilde f$ have the same volume. In summary, the quantity under question has become
\begin{equation*}
\frac{1}{\omega_\gp(K_\gp a_\gp K_\gp)} \int_{K_\gp a_\gp K_\gp} \sum_{X \in (F^n)^m \atop \mathrm{indep.}} f(N\gp_{\infty}^{-\frac{1}{nd}} D^\mathrm{tr}X h^*a_\eta) d\omega_\gp(h).
\end{equation*}

As discussed in Section 2.5 above, we may express this integral as the sum
\begin{equation} \label{eq:3-2mid1}
\frac{1}{\omega_\gp(K_\gp a_\gp K_\gp)} \sum_{h \in \mathcal L} \sum_{X \in (F^n)^m \atop \mathrm{indep.}} f(N\gp_{\infty}^{-\frac{1}{nd}} D^\mathrm{tr} Xh^*a_\eta).
\end{equation}

We make one more modification before the main computation.
\begin{lemma} \label{lemma:ltg}
For each $h \in \mathcal L$, there exists $\gamma \in \GL(n,F)$ such that
\begin{equation*}
(\gamma h^*)_\gp \in \GL(n,\cO_\gp),\ 
(\gamma a_\eta)_\nu \in (a_\eta)_\nu \cdot \GL(n,\cO_\nu) \mbox{ for all finite } \nu \neq \gp.
\end{equation*}
More precisely, let $I = \bigcap_{i=1}^k I^{(i)}$. Then there exists a subset $\mathcal R \subseteq I$ that bijects to $\mathcal R_\gp$ by reduction modulo $\gp$. For $h = h(j; a_{1,\gp}, \ldots, a_{n-j-1,\gp}) \in \mathcal L$, the corresponding $\gamma = \gamma(h)$ will be
\begin{equation*}
\gamma(j;a_1,\ldots, a_{n-j-1}) = \begin{pmatrix}
\mathrm{Id}_j &                                  &          &            & \\
                       & \tilde\pi_\gp              &           &        &             \\
                       & a_{1}      & 1       &            &  \\
                       & \vdots                       &          & \ddots & \\
                       & a_{n-j-1} &          &            & 1
\end{pmatrix},
\end{equation*}
where $a_i \mapsto a_{i,\gp}$ under the said bijection.
\end{lemma}

\begin{proof}

Consider the $\cO_F$-module map
\begin{equation*}
\iota: I \hookrightarrow I^{(i)} \rightarrow \cO_\gp,
\end{equation*}
where the first map is inclusion, and the second map is the restriction to the place $\gp$. Because $I$ is coprime to $\gp$, $\iota$ is a surjection. It also induces the $\cO_F$-module isomorphisms
\begin{equation*}
I/\gp I \cong I^{(i)}/\gp I^{(i)} \cong \cO_\gp / \gp\cO_\gp =: \F_\gp.
\end{equation*}
Choose any $\mathcal R' \subseteq I$ so that $\mathcal R'$ bijects to $\mathcal R_\gp$ under $\iota$. Next, choose $\delta \in \cO_F$ such that
\begin{equation*}
\delta_\gp \equiv 1 (\mbox{mod } \gp),\  \delta_\nu, (\delta\eta)_\nu \in \gp_\nu^{-\ord_\nu I} \cap \cO_\nu \mbox{ for all finite } \nu \neq \gp.
\end{equation*}
Let $\mathcal R := \delta \mathcal R'$. This step ensures that both $a$ and $\eta a$ are integral at all finite places for any $a \in \mathcal R$, a property we will need below.

For $h=h(j; a_{1,\gp}, \ldots, a_{n-j-1,\gp})$, we match $\gamma=\gamma(j;a_1,\ldots,a_{n-j-1,\gp})$, where $a_i$ is the (unique) element of $\mathcal R$ such that $\iota(\delta^{-1}a_i) = \iota(a_i) = a_{i,\gp} \in \mathcal R_\gp$; in particular we have $(a_i)_\gp - a_{i,\gp} \in \gp\cO_\gp$. Our choice of $\mathcal R$ allows us to check by straightforward matrix multiplications that $(\gamma h^*)_\gp \in \GL(n,\cO_\gp)$ and that $(\gamma_\nu a_\eta)_\nu \in (a_\eta)_\nu \cdot \GL(n,\cO_\nu)$ for finite $\nu \neq \gp$, as desired.
\end{proof}

Let $\mathcal M$ be the set corresponding to $\mathcal L$ by Lemma \ref{lemma:ltg}. It will be convenient to partition $\mathcal M = \bigcup_{j=0}^{n-1} \mathcal M^{(j)}$, where
\begin{equation*}
\mathcal M^{(j)} = \left\{ 
\gamma(j;a_1,\ldots,a_{n-j-1}) : a_i \in \mathcal R
\right\}.
\end{equation*}
By replacing $X$ with $X\gamma(h)$, \eqref{eq:3-2mid1} can now be rewritten as
\begin{align}
&\frac{1}{\omega_\gp(K_\gp a_\gp K_\gp)} \sum_{\gamma \in \mathcal M}\sum_{X \in (F^n)^m \atop \mathrm{indep.}} f_\f(D^\tr X\gamma h^* a_\eta) f_\infty(N(\gp)^{-\frac{1}{nd}}(D^\mathrm{tr}X\gamma a_\eta)_\infty) \notag \\
&= \frac{1}{\omega_\gp(K_\gp a_\gp K_\gp)} \sum_{\gamma \in \mathcal M}\sum_{X \in (F^n)^m\ \mathrm{indep.} \atop D^\mathrm{tr}X \in \prod_i I^{(i)}_{\eta^{-1}} \times (I^{(i)})^{n-1}} f_\infty(N(\gp)^{-\frac{1}{nd}}(D^\mathrm{tr}X\gamma a_\eta)_\infty), \label{eq:intermediate2}
\end{align}
where
\begin{equation*}
I^{(i)}_{\eta^{-1}} = I^{(i)} \cdot \prod_{\nu \in \f} \gp_\nu^{- \mathrm{ord}_\nu \eta}.
\end{equation*}

\subsection{Proof of Proposition \ref{prop:main}: the main term}

It turns out that the main contribution to \eqref{eq:intermediate2} comes from the partial sum
\begin{equation}\label{eq:he_main}
\frac{1}{\omega_\gp(K_\gp a_\gp K_\gp)} \sum_{\gamma \in \mathcal M^{(0)}}\sum_{X \in (F^n)^m\ \mathrm{indep.} \atop D^\mathrm{tr}X \in \prod_i I^{(i)}_{\eta^{-1}} \times (I^{(i)})^{n-1}} f_\infty(N(\gp)^{-\frac{1}{nd}}(D^\mathrm{tr}X\gamma a_\eta)_\infty).
\end{equation}
In this section we compute an asymptotic for \eqref{eq:he_main}.

For $X \in (F^n)^m$, let $\bar X$ be the $m \times (n-1)$ matrix obtained by removing the first column of $X$. Let us write explicitly
\begin{equation*}
X =
\begin{pmatrix}
x_{10} & x_{11} & \ldots & x_{1,n-1} \\
\vdots  &  \vdots  &          & \vdots  \\
x_{m0} & x_{m1} & \ldots & x_{m,n-1}
\end{pmatrix},
\ \bar X :=
\begin{pmatrix}
x_{11} & \ldots & x_{1,n-1} \\
\vdots  &          & \vdots  \\
x_{m1}  & \ldots & x_{m,n-1}
\end{pmatrix}.
\end{equation*}
Thanks to our assumption about the shape of $D$ (comment (i) after the statement of Theorem \ref{thm:main}), for $1 \leq i \leq m$ we may assume that the $i$-th row of $X$ is an element of $I^{(i)}_{\eta^{-1}} \times (I^{(i)})^{n-1}$, and similarly for $\bar X$. For convenient reference, we also write out
\begin{equation} \label{eq:xgamma}
X\gamma = 
\begin{pmatrix}
x_{10}\tilde\pi_\gp + \sum_{i=1}^{n-1} a_ix_{1i} & x_{11} & \ldots & x_{1,n-1} \\
\vdots  & \vdots &          & \vdots  \\
x_{m0}\tilde\pi_\gp + \sum_{i=1}^{n-1} a_ix_{mi} & x_{m1} & \ldots & x_{m,n-1}
\end{pmatrix},
\end{equation}
where $\gamma = \gamma(0;a_1,\ldots,a_{n-1}) \in \mathcal M^{(0)}$.

We first show that those $X$ for which the corresponding $\bar X$ is dependent mod $\gp$ --- more precisely, its rows, reduced mod $\gp$, are linearly dependent over $\F_\gp$ --- do not contribute at all to \eqref{eq:he_main}, provided $N\gp$ is sufficiently large.

\begin{lemma} \label{lemma:nc}
Take the notations above, and suppose $\bar X$ is dependent mod $\gp$. Then at least one row $x_i\gamma$ of $X\gamma$ satisfies $H_\infty(x_i\gamma) \gg N\gp^{1/md}$, where the implied constant is only dependent on $\eta$ and $I^{(1)}, \ldots, I^{(m)}$.
\end{lemma}
\begin{proof}
Using the surjectivity of the projection $\SL(m,\cO_F) \rightarrow \SL(m,\F_\gp)$ we can find $A \in \SL(m,\cO_F)$ such that one row of $A\bar X$ is a multiple of $\tilde\pi_\gp$, i.e. has entries in $\sum_i \gp I^{(i)}$. Then the corresponding row of $AX\gamma$ also has entries in $\sum_i \gp I^{(i)}$, but it is guaranteed to be nonzero, since the rows of $AX\gamma$ are independent. Hence, for the matrix $Y \in (\sum_i I^{(i)})^{nm}$ obtained from $AX\gamma$ by dividing that row by $\tilde\pi_\gp$,
\begin{equation*}
H_\infty(X\gamma) = H_\infty(AX\gamma) = N(\gp)H_\infty(Y) \gg_{I^{(1)},\ldots,I^{(m)}} N(\gp)
\end{equation*}
(the latter lower bound is a consequence of the reduction theory --- see e.g. \cite[Lemma 16.2]{B19}). Lemma \ref{lemma:height} now implies that $H_\infty(x_i\gamma) \gg N\gp^{1/m}$ for some $i$, where $x_i$ denotes the $i$-th row of $X$. By Lemma \ref{lemma:ideal_min}, $\rho(x_i\gamma) \in (F \otimes \R)^n$ has length $\gg N\gp^{1/md}$, and thus $\rho(x_i\gamma) \cdot (a_\eta)_\infty$ has length $\gg N\gp^{1/md}$, where the implied constant depends only on $\eta$ and $I^{(1)}, \ldots, I^{(m)}$.
\end{proof}

It follows as an immediate corollary that
\begin{equation*}
f_\infty(N\gp^{-\frac{1}{nd}}(D^\mathrm{tr}X\gamma a_\eta)_\infty) = 0
\end{equation*}
for $N\gp$ sufficiently large (independently of $X$), since $f_\infty$ is compactly supported, as desired.

This leaves those $X$ for which $\bar X$ is independent mod $\gp$. The following lemma is the crux of our proof of Proposition \ref{prop:main}.

\begin{lemma} \label{lemma:crux}
If $\bar X$ is independent mod $\gp$, then the number of $\gamma \in \mathcal M^{(0)}$ such that the vector
\begin{equation} \label{eq:eqdist_vr}
(\sum_{i=1}^{n-1} a_ix_{1i}, \ldots, \sum_{i=1}^{n-1} a_ix_{mi})^\tr \in I^{(1)} \times \cdots \times I^{(m)}
\end{equation}
represents any given coset of $I^{(1)}_{\eta^{-1}}/\gp I^{(1)}_{\eta^{-1}} \times \cdots \times I^{(m)}_{\eta^{-1}}/\gp I^{(m)}_{\eta^{-1}}$ is exactly $N\gp^{n-1-m}$.
\end{lemma}
\begin{proof}
Observe that, since the rows of $\bar X$ are independent mod $\gp$, the map
\begin{equation*}
(a_1, \ldots, a_{n-1})^\tr \mapsto \bar X(a_1, \ldots, a_{n-1})^\tr = (\sum_{i=1}^{n-1} a_ix_{1i}, \ldots, \sum_{i=1}^{n-1} a_ix_{mi})^\tr
\end{equation*}
is a surjective linear map from $\F_\gp^{n-1}$ to $\F_\gp^m$, so each image has $N\gp^{n-1-m}$ preimages. The proof will be complete if we identify the domain with $\mathcal R^{n-1}$ and the codomain with $I^{(1)}_{\eta^{-1}}/\gp I^{(1)}_{\eta^{-1}} \times \cdots \times I^{(m)}_{\eta^{-1}}/\gp I^{(m)}_{\eta^{-1}}$.

Indeed, restricting $\mathcal R$ modulo $\gp$ we obtain every element of $\F_\gp$.  Also, thanks to our choice of $\mathcal R$ in the proof of Lemma \ref{lemma:ltg}, $(\eta a)_\nu$ is integral for any finite $\nu$ and $a \in \mathcal R$, and thus
\begin{equation*}
(\sum_{i=1}^{n-1} a_ix_{1i}, \ldots, \sum_{i=1}^{n-1} a_ix_{mi})^\tr \in I^{(1)}_{\eta^{-1}} \times \cdots \times I^{(m)}_{\eta^{-1}}
\end{equation*}
holds for any $a_1, \ldots, a_{n-1} \in \mathcal R$. Taking \eqref{eq:eqdist_vr} mod $\gp$ completes the proof.
\end{proof}

Lemma \ref{lemma:crux} implies that, as $\gamma$ runs over $\mathcal M^{(0)}$, and $(x_{10}, \ldots, x_{m0})^\tr$ runs over $\prod_i I_{\eta^{-1}}^{(i)}$, the first column of $X\gamma$ (see \eqref{eq:xgamma} above) hits each and every element of $\prod_i I_{\eta^{-1}}^{(i)}$ exactly $N\gp^{n-1-m}$ times. Therefore, combined with the fact that only those $X$ for which $\bar X$ is independent mod $\gp$ contribute, \eqref{eq:he_main} can be rewritten as
\begin{equation*}
\frac{N\gp^{n-1-m}}{\omega_\nu(K_\gp a_\gp K_\gp)} \sum_{X \in (F^n)^m \atop {D^\mathrm{tr}X \in \prod_i I^{(i)}_{\eta^{-1}} \times (I^{(i)})^{n-1} \atop \rk_\gp \bar X = m }} f_\infty(N\gp^{-\frac{1}{nd}}(D^\tr X a_\eta)_\infty).
\end{equation*}
Here $\rk_\gp$ means the rank over $\F_\gp$. Since
\begin{equation*}
\frac{N\gp^{n-1-m}}{\omega_\nu(K_\gp a_\gp K_\gp)} = \frac{1}{N\gp^m}(1 + o(1)),
\end{equation*}
there is no harm in considering
\begin{equation} \label{eq:mainest1}
\frac{1}{N\gp^m} \sum_{X \in (F^n)^m \atop {D^\mathrm{tr}X \in \prod_i  I^{(i)}_{\eta^{-1}} \times (I^{(i)})^{n-1} \atop \rk_\gp \bar X = m }} f_\infty(N\gp^{-\frac{1}{nd}}(D^\tr X a_\eta)_\infty)
\end{equation}
instead. We claim that, up to an error that vanishes as $N\gp \rightarrow \infty$, this is equal to
\begin{equation} \label{eq:mainest2}
\frac{1}{N\gp^m} \sum_{X \in (F^n)^m \atop D^\mathrm{tr}X \in \prod_i  I^{(i)}_{\eta^{-1}} \times (I^{(i)})^{n-1}} f_\infty(N\gp^{-\frac{1}{nd}}(D^\mathrm{tr}X a_\eta)_\infty).
\end{equation}

Let us set aside the claim for now, and explain first how the right-hand side of \eqref{eq:inprop} comes about from \eqref{eq:mainest2}.
\eqref{eq:mainest2} can be computed by rewriting it in the classical language. Applying Lemma \ref{lemma:index} to the support $S \subseteq \A_\f^{nm}$ of $f_\f((D^\tr Xa_\eta)_\f)$, the image under $\rho$ of the elements $X \in (F^n)^m$ such that $D^\tr X \in \prod_i  I^{(i)}_{\eta^{-1}} \times (I^{(i)})^{n-1}$ forms a lattice $L \subseteq (F \otimes \R)^{nm}$ of rank $mnd$ and determinant
\begin{equation*}
|\Delta_F|^{mn/2}\left(\int_{\A_\f^{mn}} f_{\f}((D^\tr X a_\eta)_\f) d\alpha_\f^{mn}\right)^{-1}.
\end{equation*}
Accordingly we write \eqref{eq:mainest2} as
\begin{equation*}
\frac{1}{N\gp^{m}} \sum_{X \in N\gp^{-\frac{1}{nd}}L} f_\infty((D^\tr X a_\eta)_\infty),
\end{equation*}
or more suggestively,
\begin{equation*}
\frac{1}{N\gp^{m}\det (N\gp^{-\frac{1}{nd}}L)} \sum_{X \in N\gp^{-\frac{1}{nd}}L} \det (N\gp^{-\frac{1}{nd}}L)f_\infty((D^\tr X a_\eta)_\infty).
\end{equation*}
Note that $\det (N\gp^{-\frac{1}{nd}}L) = N\gp^{-m}\det L$. This is a Riemann sum, and $f_\infty$ is Riemann integrable, so we are in a situation to appeal to the following well-known principle.
\begin{proposition} \label{prop:riemann}
Suppose $\Lambda \subseteq \R^n$ is a (full-rank) lattice, and $\varphi$ is a Riemann integrable function on $\R^n$. Then
\begin{equation*}
\sum_{x \in \varepsilon \Lambda} \det(\varepsilon \Lambda)\varphi(x) \rightarrow \int \varphi dx
\end{equation*}
as $\varepsilon \rightarrow 0$.
\end{proposition}
We conclude that \eqref{eq:mainest2} is equal to, in the $N\gp$ limit,
\begin{align*}
& |\Delta_F|^{-mn/2}\int_{\A_\f^{mn}} f_{\f}((D^\tr X a_\eta)_\f) d\alpha_\f^{mn} \int_{\A_\infty^{mn}}f_\infty((D^\tr X a_\eta)_\infty)d\alpha_\infty^{mn} \\
&= \int_{(\A_F^n)^m} f(D^\tr X) \|\eta\|^{-m} d\alpha_F^{mn} \\
&= \int_{(\A_F^n)^m} f(D^\tr X) d\alpha_F^{mn},
\end{align*}
(recall $\|\eta\|=1$) matching the claim of Proposition \ref{prop:main}.

We return to showing the equivalence of \eqref{eq:mainest1} and \eqref{eq:mainest2} up to a small error. We start by claiming that the condition $\rk_\gp \bar X = m$ in \eqref{eq:mainest1} can be replaced by $\rk \bar X = m$. Since $\rk_\gp \bar X \leq \rk \bar X$ always, it suffices to consider the case $\rk_\gp \bar X < \rk \bar X = m$. But by a similar (shorter) argument as in the proof of Lemma \ref{lemma:nc}, such $\bar X$, and thus $X$ too, has a row whose image via $\rho$ has length $\gg N\gp^{1/md}$, and thus $X$ contributes zero to the sum for $N\gp$ sufficiently large.

Next we claim that the condition $\rk \bar X = m$ may be discarded altogether. To this end, consider those $X$ such that $\bar X$ is dependent. Without loss of generality, we may assume that the last row of $\bar X$ is dependent on the other rows. Then $x_m$, the last row of $X$, lies in the rank $\leq m$ submodule
\begin{equation*}
N =  \left(I^{(m)}_{\eta^{-1}} \times (I^{(m)})^{n-1}\right) \cap (Fe_1 \oplus Fx_1 \oplus \ldots \oplus Fx_{m-1})
\end{equation*}
of $ I^{(m)}_{\eta^{-1}} \times (I^{(m)})^{n-1}$, where $e_1 = (1, 0, \ldots, 0) \in F^n$. The sum \eqref{eq:mainest1} restricted to all such $X$ is crudely bounded by
\begin{equation} \label{eq:error_main1}
\frac{1}{N\gp^m}\sum_{\hat X \in \prod_{i=1}^{m-1} I^{(i)}_{\eta^{-1}} \times (I^{(i)})^{n-1}} \prod_{i \neq m} f_\infty^{(i)} (N\gp^{-\frac{1}{nd}}(\hat Xa_\eta)_\infty) \sum_{x_m \in N} f_\infty^{(m)}(N\gp^{-\frac{1}{nd}}(x_ma_\eta)_\infty).
\end{equation}
The inner sum counts the number of the vectors of the rank $\leq md$ lattice $N\gp^{-\frac{1}{nd}}\rho(N) \cdot (a_\eta)_\infty$ inside a bounded set, and thus by an appropriate adaptation of Proposition \ref{prop:riemann}, it is of size $O_{f,I,\eta}(N\gp^{m/n})$. Similarly, the outer sum is $O_{f, I, \eta}(N\gp^{m-1})$, and thus \eqref{eq:error_main1} is of size $O_{f,I,\eta}(N\gp^{m/n-1}) = o(1)$, as desired.

\subsection{Proof of Proposition \ref{prop:main}: error terms}

To complete the proof of Proposition \ref{prop:main} and thus of Theorem \ref{thm:main}, it remains to estimate the intended error terms
\begin{equation}\label{eq:he_error1}
\frac{1}{\omega_\gp(K_\gp a_\gp K_\gp)} \sum_{\gamma \in \mathcal M^{(j)}}\sum_{X \in (F^n)^m\ \mathrm{indep.} \atop D^\mathrm{tr}X \in \prod_i I^{(i)}_{\eta^{-1}} \times (I^{(i)})^{n-1}} f_\infty(N\gp^{-\frac{1}{nd}}(D^\mathrm{tr}X\gamma a_\eta)_\infty)
\end{equation}
for each $1 \leq j \leq n-1$. Here we will show that \eqref{eq:he_error1} vanishes as $N\gp \rightarrow \infty$.
To this end, it suffices to assume $m=k$, and work with the simpler
\begin{equation} \label{eq:he_error}
\frac{1}{N\gp^{n-1}} \sum_{\gamma \in \mathcal M^{(j)}}\sum_{X \in (J^n)^m \atop \mathrm{indep.}} g_\infty(N\gp^{-\frac{1}{nd}}(X\gamma)_\infty),
\end{equation}
where $J = I_{\eta^{-1}} + \sum_{i=1}^k I^{(i)}$ and $g_\infty(X) = |f_\infty(X(a_\eta)_\infty)|$, since this is no smaller than \eqref{eq:he_error1}, possibly up to a constant factor --- recall that $f_\infty$ is bounded by assumption.

For $\gamma = \gamma(j;a_1,\ldots,a_{n-j-1}) \in \mathcal M^{(j)}$ and
\begin{equation*}
X =
\begin{pmatrix}
x_{10} & x_{11} & \ldots & x_{1,n-1} \\
\vdots  &  \vdots  &          & \vdots  \\
x_{m0} & x_{m1} & \ldots & x_{m,n-1}
\end{pmatrix},
\end{equation*}
$X\gamma$ is of the form
\begin{equation*}
\begin{pmatrix}
x_{10} & \cdots & x_{1,j-1} & x_{1j}\tilde\pi_\gp + \sum_{i=j+1}^{n-1} a_{i-j}x_{1i} & x_{1,j+1} & \cdots & x_{1,n-1} \\
\vdots &            &    \vdots   &   \vdots                                                                      &  \vdots     &            &   \vdots    \\
x_{m0} & \cdots & x_{m,j-1} & x_{mj}\tilde\pi_\gp + \sum_{i=j+1}^{n-1} a_{i-j}x_{mi} & x_{m,j+1} & \cdots & x_{m,n-1}
\end{pmatrix}.
\end{equation*}
This time we let
\begin{equation*}
\bar{X} = \begin{pmatrix}
x_{1, j+1} & \cdots & x_{1,n-1} \\
\vdots &   & \vdots \\
x_{m,j+1} & \cdots & x_{m,n-1}
\end{pmatrix}.
\end{equation*}

We first consider the partial sum of \eqref{eq:he_error} over those $X$ for which $\rk_\gp \bar X = m$.
Then we can repeat the same argument as in the previous section (specifically, Lemma \ref{lemma:crux} and the paragraph following it) to prove that the restricted sum is equal to
\begin{equation*}
\frac{1}{N\gp^{j+m}} \sum_{X \in (J^n)^m\ \mathrm{indep.} \atop \mathrm{rk}_\gp\, \bar X = m } g_\infty(N\gp^{-\frac{1}{nd}}X_\infty);
\end{equation*}
the core fact is that $\bar X$ mod $\gp$ induces a surjective map $\F_\gp^{n-j-1} \rightarrow \F_\gp^m$.
Dropping the rank and independence conditions on $X$, this is a sum over a lattice in an $mnd-$dimensional Euclidean space of determinant $O_{g, J}(N\gp^{-\frac{mnd}{nd}})= O_{g, J}(N\gp^{-m})$. Thus we see that, by Proposition \ref{prop:riemann}, it has size $O_{g, J}(N\gp^{-j})$.

It remains to consider those $X$ for which $\mathrm{rk}_\gp\, \bar X = m' < m$. If $m' = 0$, then $\bar X \in (\gp J)^{(n-j-1)m}$, and \eqref{eq:he_error} restricted to all such $X$ is equal to
\begin{equation*}
\frac{1}{N\gp^{j}} \sum_{X \in J^{jm} \times (\gp J)^{(n-j)m}\ \mathrm{indep.}} g_\infty(N\gp^{-\frac{1}{nd}}X_\infty).
\end{equation*}
As can be seen from the proof of Lemma \ref{lemma:height}, nonzero vectors of $\rho(\gp J)$ has length $\gg_J N\gp^{1/d}$. Hence, for $N\gp$ sufficiently large, this sum is really taken over $J^{jm}$ scaled by $N\gp^{-1/nd}$. By Proposition \ref{prop:riemann}, it follows that it is of size $O_{g,J}(N\gp^{-j(1-m/n)})$.


If $m' \neq 0$, then we divide further into two cases. First suppose that $j < m - m'$. Denoting by $\tilde X$ the $m \times (n-j)$ matrix consisting of the last $n-j$ columns of $X$, and by $\tilde \gamma$ the $(n-j) \times (n-j)$ second block of $\gamma$.
Arguing similarly as in the proof of Lemma \ref{lemma:nc}, we can find $A \in \SL(m,\cO_F)$ such that precisely $m-m'$ rows of $A\tilde X\tilde\gamma$ have all their entries in $\gp J$. However, since $\mathrm{rk}_F\, \tilde X = \mathrm{rk}_F\, \tilde X \tilde\gamma > m'$ is forced, at least one of those rows must be nonzero, and thus again as in the proof of Lemma \ref{lemma:nc}, we may show that for some row $x_i$ of $X$, $\rho(x_i\gamma)$ has length $\gg_J N\gp^{1/md}$. It follows that such $X$ does not contribute to \eqref{eq:he_error}.

It remains to consider the case $j \geq m-m'$, which is a bit more involved than the previous cases. Without loss of generality, let us assume that it is the first $m'$ vectors of $\bar X$ that are independent mod $\gp$. Also let us write
\begin{equation*}
g_1 = \prod_{i=1}^{m'} g_\infty^{(i)}, g_2 = \prod_{i=m'+1}^{m} g_\infty^{(i)},
\end{equation*}
where $g_\infty^{(i)}(X) = f_\infty^{(i)}(Xa_\eta)$. Then the restriction of \eqref{eq:he_error} to all the $X$ with $\rk_\gp \bar X = m'$ is bounded by
\begin{equation*}
\frac{1}{N\gp^{n-1}} \sum_{\gamma \in \mathcal M^{(j)}} \sum_{Y_1 \in (J^{n})^{m'}\ \mathrm{indep}. \atop \rk_\gp \bar Y_1 = m'} g_1(N\gp^{-\frac{1}{nd}}(Y_1\gamma)_\infty) \sum_{Y_2 \in (J^{n})^{m-m'} \atop \bar Y_2\,\mathrm{mod}\,\gp \in \spn_\gp(\bar Y_1)} g_2(N\gp^{-\frac{1}{nd}}(Y_2\gamma)_\infty),
\end{equation*}
where $\bar Y$ denotes the matrix formed by the last $n-j-1$ columns of $Y$, and $\bar Y_2\,\mathrm{mod}\,\gp \in \spn_\gp(\bar Y_1)$ means that each row of $\bar Y_2$ taken mod $\gp$ is contained in the $\F_\gp$-subspace of $\F_\gp^{n-j-1}$ spanned by the rows of $\bar Y_1$. Again using the fact that $\bar Y_1$ induces a surjective mapping $\F_\gp^{n-j-1} \rightarrow \F_\gp^{m'}$ (note we must have $n-j-1 \geq m'$ indeed), and adapting Lemma \ref{lemma:crux} and its subsequent argument, this is equal to
\begin{equation} \label{eq:he_error2}
\frac{1}{N\gp^{j+m'}} \sum_{Y_1 \in (J^{n})^{m'}\ \mathrm{indep}. \atop \rk_\gp \bar Y_1 = m'} g_1(N\gp^{-\frac{1}{nd}}Y_{1,\infty}) \sum_{Y_2 \in J^{j(m-m')} \times \kappa(Y_1)} g_2(N\gp^{-\frac{1}{nd}}Y_{2,\infty}),
\end{equation}
where we write
\begin{equation*}
\kappa(Y) = (\gp J)^{(n-j)(m-m')} + (\spn(\tilde Y) \cap J^{(n-j)})^{(m-m')} \subseteq (F^{(n-j)})^{(m-m')},
\end{equation*}
with $\tilde Y$ being the matrix
\begin{equation*}
\begin{pmatrix}
\sum_{i=j+1}^{n-1} a_{i-j}x_{1i} & x_{1, j+1} & \cdots & x_{1,n-1} \\
\vdots & \vdots &   & \vdots \\
\sum_{i=j+1}^{n-1} a_{i-j}x_{m'i} & x_{m',j+1} & \cdots & x_{m',n-1}
\end{pmatrix},
\end{equation*}
and $\spn(\tilde Y)$ the $F$-span of the rows of $\tilde Y$.

Similarly as in the $m' = 0$ case, as $N\gp$ grows, the sum over $Y_2$ in \eqref{eq:he_error2} is really the sum over a $\cO_F$-module of $F$-rank $(m-m')j + (m-m')\rk_F\,\tilde Y = (m-m')(j+m')$, scaled by $N(\gp)^{-1/nd}$. Therefore, by (two applications of) Proposition \ref{prop:riemann}, \eqref{eq:he_error2} is of the size
\begin{equation*}
O_{g,j}\left(N\gp^{-j-m'+m'+\frac{1}{n}(m-m')(j+m')}\right) \leq O_{g,J}(N\gp^{-\frac{j}{n}}),
\end{equation*}
which vanishes as $N\gp \rightarrow \infty$.

\section{Extensions of Theorem \ref{thm:main}}

\subsection{Primitivity}

For positive integers $m \leq n$ and a finite place $\nu$ of $F$, we say an $m \times n$ matrix $X_\nu$ with entries in $\cO_\nu$ is primitive (at $\nu$) if $X_\nu$ can be completed to an element of $\GL(n,\cO_\nu)$. For $X$ with entries in multiple places, such as $F, \A_\f,$ or $\A_F$, we say $X$ is primitive if it is primitive at each of those places. On the other hand, we say the vectors $x_1, \ldots, x_m \in \cO_\nu^n$ span a primitive lattice (in $\cO_\nu^n$) if their $\cO_\nu$-span coincides with the intersection of $\cO_\nu^n$ and their $F_\nu$-span; for vectors with entries in multiple places, we extend the definition in the same manner as earlier. The two concepts can be easily seen to coincide:
\begin{lemma}\label{lemma:prim_mat}
Let $m \leq n$ be positive integers, and $\nu \in \f$. $X \in \mathrm{Mat}_{m \times n}(\cO_\nu)$ is primitive at $\nu$ if and only if its row vectors $x_1, \ldots, x_m$ span a primitive lattice.
\end{lemma}
\begin{proof}
Suppose first that $X$ is primitive. Let $\tilde X \in \GL(n,\cO_\nu)$ be its completion. Any $v \in \cO_\nu^n \cap \spn_{F_\nu}(x_1,\ldots,x_m)$ has an expression $v = c_1x_1 + \ldots c_mx_m$ for $c_i \in F_\nu$. Notice that $v\tilde X^{-1} = (c_1, \ldots, c_m, 0, \ldots, 0) \in \cO_\nu^n$; this proves that the $x_i$'s span a primitive lattice.

Suppose conversely that the $x_i$'s span a primitive lattice. We can complete it to a $\cO_\nu$-basis of $\cO_\nu^n$ by adding $n-m$ elementary vectors, that is, vectors of the form $e_i \in \cO_\nu^n$, whose $i$-th entry is $1$ and the rest are zero. Without loss of generality, we may assume we can choose to add $e_{m+1}, \ldots, e_n$. Accordingly, we claim that the matrix
\begin{equation*}
P =
\begin{pmatrix}
  & - & x_1 & - &  \\
  &  & \vdots  &  &  \\
  & -  & x_m  & - &  \\
  & -  & e_{m+1}  & - &  \\
  &   & \vdots  &  &  \\
  & -  & e_n  & - &  \\
\end{pmatrix}
\end{equation*}
has unit determinant. It suffices to demonstrate $Q \in \mathrm{Mat}_{n \times n}(\cO_\nu)$ such that $QP = \mathrm{Id}_n$. Thanks to the assumptions, for any $1 \leq i \leq n$ we can write $e_i = \sum_{j\leq m} c_{ij}x_j + \sum_{j>m} c_{ij}e_j$ for some $c_{ij} \in \cO_\nu$; we take $Q = (c_{ij})_{1 \leq i,j \leq n}$.
\end{proof}

In case $F = \Q$ for example, our notion of primitivity coincides with the usual notion of primitivity for $X \in \mathrm{Mat}_{m \times n}(\Z)$, which is that it can be completed to an element of $\GL(n,\Z)$. The direction (usual notion) $\Rightarrow$ (our notion) is clear. For the other direction: if $X$ cannot be completed to an element of $\GL(n,\Z)$, then $\rk_{\F_p} X < m$ for some prime $p$, and hence there cannot exist an element of $\GL(n,\Z_p)$ completing $X_p$.

By the theory of the Smith normal form, each $X \in \mathrm{Mat}_{m \times n}(\cO_\nu)$ is of the form
\begin{equation*}
X = \gamma \cdot \mathrm{diag}(\pi_\nu^{a_1}, \ldots, \pi_\nu^{a_m}) \cdot P,
\end{equation*}
where $\gamma \in \GL(m, \cO_\nu)$, $P \in \mathrm{Mat}_{m \times n}(\cO_\nu)$ is primitive, and $0 \leq a_1 \leq \ldots \leq a_n$, where the $a_i$'s are determined uniquely by $X$. We allow $a_i$ to be $\infty$, interpreting $\pi_\nu^\infty = 0$. If any $a_i = \infty$, then the row vectors of $X$ are linearly dependent over $F_\nu$, and the converse also holds. On the other end of the spectrum, $a_1 = \ldots = a_m = 0$ if and only if $X$ is primitive. Moreover, we have the following lemma.

\begin{lemma}\label{lemma:repr}
Let $X \in \mathrm{Mat}_{m \times n}(\cO_\nu)$. Then we can write
\begin{equation*}
X = hP,
\end{equation*}
where  $P \in \mathrm{Mat}_{m \times n}(\cO_\nu)$ is primitive, and $h \in K_\nu\,\mathrm{diag}(\pi_\nu^{a_1}, \ldots, \pi_\nu^{a_m}) K_\nu$ for some $0 \leq a_1 \leq \ldots \leq a_m$. Furthermore, the coset $h K_\nu$ is uniquely determined by $X$.
\end{lemma}
\begin{proof}
The existence of such $P$ and $h$ is clear from the preceding discussion. For the uniqueness claim, suppose $hP = h'P'$
for some primitive $P,P'$ and $h, h' \in K_\nu \mathrm{diag}(\pi_\nu^{a_1}, \ldots, \pi_\nu^{a_m}) K_\nu$. Since the rows of $P$ and $P'$ have the same $F_\nu$-span and therefore the same $\cO_\nu$-span, we can actually complete them into elements $\tilde P, \tilde P'$ of $\GL(n, \cO_\nu)$ whose last $n-m$ rows are identical (cf. proof of Lemma \ref{lemma:prim_mat} above).

Let $\tilde h$ (resp. $\tilde h'$) be the block matrix whose first $m \times m$ block is $h$ (resp. $h'$), and the next (and last) $(n-m) \times (n-m)$ block is the identity matrix. We have $\tilde h, \tilde h' \in \GL(n, \cO_\nu)$, and also
\begin{equation*}
\tilde h \tilde P = \tilde h' \tilde P',
\end{equation*}
from which it is immediate that $h^{-1}h' \in K_\nu$ and thus $hK_\nu = h'K_\nu$.
\end{proof}

In this section, we use the Hecke operators $\mathcal T$ to implement an inclusion-exclusion argument that expands the family of functions for which Theorem \ref{thm:main} applies. The propositions below serve as the starting point.

\begin{proposition}\label{prop:indprim}
Choose $\mathbf s \subseteq \f$. Let $f_\mathbf s$ be the characteristic function of $\prod_{\nu \in \mathbf s} (\cO_\nu^n)^m$, and $f_{\mathrm{pr},\mathbf s}$ be the characteristic function of the set of primitive $m \times n$ matrices in $\prod_{\nu \in \mathbf s} (\cO_\nu^n)^m$. Then
\begin{equation*}
f_\mathbf s(X) = \sum_I \mathcal T(I)f_{\mathrm{pr},\mathbf s}(X),
\end{equation*}
where $I$ runs over all nonzero ideals generated by $\mathbf s$.

Conversely, for $0 \leq i \leq m$ and $\nu \in \f$, let
\begin{equation*}
\mathcal T^{(i)}(\nu) = \mathcal T(\underbrace{1, \cdots, 1}_{m-i}, \underbrace{\nu, \cdots, \nu}_{i}).
\end{equation*}
Then
\begin{equation*}
f_\mathrm{pr,\mathbf s}(X) = \prod_{\nu \in \mathbf s} \left(\sum_{i=0}^m (-1)^iN(\nu)^{i(i-1)/2}\mathcal T^{(i)}(\nu) \right)f_\mathbf s(X).
\end{equation*}

\end{proposition}
\begin{proof}
The former equality is an immediate consequence of Lemma \ref{lemma:repr}. The latter follows from \cite[Theorem 3.21]{Shi71} --- originally due to Tamagawa (\cite{Tam63}) --- which states that
\begin{equation*}
\sum_{i=0}^m (-1)^iN(\nu)^{i(i-1)/2}\mathcal T^{(i)}(\nu) \mbox{ and } \sum_{i=0}^\infty \mathcal T(\nu^i)
\end{equation*}
are inverse operators to one another.
\end{proof}

\begin{proposition}\label{prop:prim_expn}
Choose $\mathbf s \subseteq \f$. Let $f = f_{\f} f_\infty$ be a function $(\A_F^n)^m \rightarrow \R$ such that $f_\infty$ is an integrable function on $(\A_\infty^n)^m$, and $f_\f$ is the characteristic function of $\prod_{\nu \in \f} (\cO_\nu^n)^m$. Also, let $f_\mathrm{pr} = f_{\mathrm{pr}, \mathbf s} f_\infty$, where $f_{\mathrm{pr}, \mathbf s}$ is the restriction of $f_\f$ to those $X$ such that $X_\nu$ is primitive for $\nu \in \mathbf s$. Choose any $g \in \GL(n,\A_F)$. Then
\begin{equation*}
\sum_{X \in (F^n)^m \atop \mathrm{indep.}} f(Xg) = \sum_{X \in (F^n)^m \atop \mathrm{indep.}} \sum_I \mathcal T(I) f_\mathrm{pr}(Xg),
\end{equation*}
where $I$ runs over all nonzero integral ideals generated by $\mathbf s$. Moreover,
\begin{equation*}
\sum_{X \in (F^n)^m \atop \mathrm{indep.}} f_\mathrm{pr}(Xg) = \sum_{X \in (F^n)^m \atop \mathrm{indep.}} \prod_{\nu \in \mathbf s} \left(\sum_{i=0}^m (-1)^iN(\nu)^{i(i-1)/2}\mathcal T^{(i)}(\nu) \right)f(Xg).
\end{equation*}
\end{proposition}
\begin{proof}
Both follow from Proposition \ref{prop:indprim} by summing over all $Xg$ such that $X \in (F^n)^m$ is linearly independent.
\end{proof}

\subsection{Rogers integral formula, primitive case}

We now prove the main results of this section.

\begin{theorem}\label{thm:meanhecke}
Let $f = f_\f f_\infty$, where $f_\infty$ is Riemann integrable on $(\A_\infty^n)^m$ that is bounded and compactly supported, and $f_\f$ is the characteristic function of $\prod_{\nu \in \f} (\cO_\nu^n)^m$. Also, let $h \in \GL(m,\A_\f)$ be such that $h_\nu$ has entries in $\cO_\nu$, and $h_\nu = \mathrm{Id}$ for all but finitely many $\nu \in \f$. Then
\begin{equation*}
\int_{\mathcal X_n} \sum_{X \in (F^n)^m \atop \mathrm{indep.}} f(h^{-1}Xg) d\mu_n = \int_{(\A_F^n)^m} f(h^{-1}X) d\alpha_F^{nm}.
\end{equation*}
As a consequence, we have
\begin{equation*}
\int_{\mathcal X_n} \sum_{X \in (F^n)^m \atop \mathrm{indep.}} \mathcal Tf(Xg) d\mu_n = \int_{(\A_F^n)^m} \mathcal Tf(X) d\alpha_F^{nm}
\end{equation*}
for any element $\mathcal T$ in the Hecke ring.
\end{theorem}
\begin{proof}
For each $\nu \in \f$, take the Smith normal form of $h_\nu$; that is, $h_\nu = \gamma_\nu a_\nu \delta_\nu$, where $\gamma_\nu, \delta_\nu \in K_\nu$ and $a_\nu$ is a diagonal matrix. Then $(h^{-1}Xg)_\nu \in (\cO_\nu^n)^m$ if and only if $Xg \in \gamma_\nu a_\nu (\cO_\nu^n)^m$. Now use the strong approximation to find $k \in \SL(m,\cO_F)$ such that $k_\nu\gamma_\nu = l_\nu \diag(\eta_\nu, 1, \ldots, 1)$ for $\eta_\nu = \det \gamma_\nu \in \cO_\nu^*$ and $l_\nu \in \SL(m, \cO_\nu)$ sufficiently close to the identity so that it fixes the set $a_\nu (\cO_\nu^n)^m$. Then, writing $a = \prod_{\nu \in \f} a_\nu$, we have
\begin{equation*}
\sum_{X \in (F^n)^m \atop \mathrm{indep.}} f(h^{-1}Xg) = \sum_{X \in (F^n)^m \atop \mathrm{indep.}} f(h^{-1}k^{-1}Xg) = \sum_{X \in (F^n)^m \atop \mathrm{indep.}} f_\f(a^{-1}Xg)f_\infty(k^{-1}Xg).
\end{equation*}
Therefore Theorem \ref{thm:main} applies, and we conclude that it equals
\begin{equation*}
\int_{(\A_F^n)^m} f(a^{-1}X) d\alpha_F^{nm} =\int_{(\A_F^n)^m} f(h^{-1}X) d\alpha_F^{nm}
\end{equation*}
as desired.
\end{proof}

\begin{theorem}\label{thm:prim}
Choose $\mathbf s \subseteq \f$.
Let $f_{\mathrm{pr}} = f_{\mathrm{pr},\mathbf s}f_{\infty}$ be a function $(\A_F^n)^m \rightarrow \R$, such that $f_\infty$ is Riemann integrable, bounded and compactly supported, and $f_{\mathrm{pr},\mathbf s}$ is the characteristic function of the set of $m \times n$ matrices over $\A_\f$ primitive over all $\nu \in \mathbf s$. Then
\begin{equation*}
\int_{\mathcal X_n} \sum_{X \in (F^n)^m \atop \mathrm{indep.}} f_\mathrm{pr}(Xg) d\mu_n = \int_{(\A_F^n)^m} f_{\mathrm{pr}}(X) d\alpha_F^{mn} = \prod_{i=0}^{m-1} \zeta_\mathbf s(n-i)^{-1} \int_{(\A_\infty^n)^m} f_\infty(X) d\alpha_\infty^{mn},
\end{equation*}
where
\begin{equation*}
\zeta_\mathbf s(s) = \prod_{\nu \in \mathbf s} (1-N(\nu)^{-s})^{-1}.
\end{equation*}
\end{theorem}

\begin{proof}
Write $f = f_\f f_\infty$, where $f_\f$ is the characteristic function of $\prod_{\nu \in \f} (\cO_\nu^n)^m$. By Proposition \ref{prop:prim_expn}, we have
\begin{equation*}
\int_{\mathcal X_n} \sum_{X \in (F^n)^m \atop \mathrm{indep.}} f_\mathrm{pr}(Xg) d\mu = \int_{\mathcal X_n} \sum_{X \in (F^n)^m \atop \mathrm{indep.}} \prod_{\nu \in \mathbf s} \left(\sum_{i=0}^m (-1)^iN(\nu)^{i(i-1)/2}\mathcal T^{(i)}(\nu) \right)f(Xg) d\mu.
\end{equation*}
For $N > 0$, let $\mathbf s_N$ be the set of all finite primes $\nu$ such that $N(\nu) < N$. Using the dominated convergence theorem, the right-hand side can be rewritten as
\begin{equation*}
\lim_{N \rightarrow \infty} \int_{\mathcal X_n} \sum_{X \in (F^n)^m \atop \mathrm{indep.}} \prod_{\nu \in \mathbf s_N} \left(\sum_{i=0}^m (-1)^iN(\nu)^{i(i-1)/2}\mathcal T^{(i)}(\nu) \right)f(Xg) d\mu.
\end{equation*}
The product over $\nu$ of the $\mathcal T$'s here is finite, and is thus an element of the Hecke ring. Hence Theorem \ref{thm:meanhecke} implies that the above is equal to
\begin{align*}
&\lim_{N \rightarrow \infty} \int_{(\A_F^n)^m} \prod_{\nu \in \mathbf s_N} \left(\sum_{i=0}^m (-1)^iN(\nu)^{i(i-1)/2}\mathcal T^{(i)}(\nu) \right)f(X) d\alpha_F^{nm} \\
&= \lim_{N \rightarrow \infty} \int_{(\A_F^n)^m} \prod_{\nu \in \mathbf s_N} \left(\sum_{i=0}^m (-1)^iN(\nu)^{i(i-1)/2}\mathcal T^{(i)}(\nu) \right)\sum_I \mathcal T(I)f_\mathrm{pr}(X) d\alpha_F^{nm} \\
&= \int_{(\A_F^n)^m} f_\mathrm{pr}(X) d\alpha_F^{nm}.
\end{align*}

It remains to compare the volumes of $f$ and $f_\mathrm{pr}$. We have
\begin{align*}
\int_{(\A_F^n)^m} f(X) d\alpha_F^{nm} &= \int_{(\A_F^n)^m} \sum_{I} \mathcal T(I) f_\mathrm{pr}(X) d\alpha_F^{nm} \\
& = \sum_I Q^{(m)}(I)N(I)^{-n} \int_{(\A_F^n)^m} f_\mathrm{pr}(X) d\alpha_F^{nm},
\end{align*}
where $I$ runs over all nonzero ideals generated by $\mathbf s$, and $Q^{(m)}(I) = \deg \mathcal T(I)$. Proposition \ref{prop:zeta} below then completes the proof of the theorem.

\end{proof}

\begin{proposition}\label{prop:zeta}
Write $Q^{(m)}(I) = \deg \mathcal T(I)$, and 
\begin{equation*}
\zeta_\mathbf s(s) = \prod_{\nu \in \mathbf s} (1-N(\nu)^{-s})^{-1}
\end{equation*}
as before. Then
\begin{equation*}
\sum_I Q^{(m)}(I)N(I)^{-n} = \zeta_\mathbf s(n) \zeta_\mathbf s(n-1) \cdots \zeta_\mathbf s(n-m+1).
\end{equation*} 
\end{proposition}
\begin{proof}
It suffices to prove the ``local'' version of the proposition, which goes as follows. Fix $\nu \in \f$, and write $q = N(\nu)$, $Q^{(m)}_\nu(k) = \deg \mathcal T(\nu^k)$, and $\zeta_\nu(k) = (1-q^{-k})^{-1}$. Then we want to prove
\begin{equation}\label{eq:zeta_local}
\sum_{k=0}^\infty Q^{(m)}_\nu(k)q^{-kn} = \zeta_\nu(n)\zeta_\nu(n-1) \cdots \zeta_\nu(n-m+1).
\end{equation}

Later we will prove that
\begin{equation}\label{eq:Q}
Q^{(m)}_\nu(k) = \prod_{i=1}^{m-1} \frac{q^{k+i} - 1}{q^i - 1}.
\end{equation}
Let us assume this for now. We will prove \eqref{eq:zeta_local} by induction on $m$. For $m=1$, it is trivially true. For the induction step, start by rewriting the left-hand side of \eqref{eq:zeta_local} as
\begin{equation*}
\frac{1}{q^{m-1}-1} \sum_{k=0}^\infty \frac{q^{k+m-1}-1}{q^{kn}} \cdot Q^{(m-1)}_\nu(k).
\end{equation*}
By the induction hypothesis, this equals
\begin{align*}
&= \frac{1}{q^{m-1}-1}\left(q^{m-1}\zeta_\nu(n-1)\cdots\zeta_\nu(n-m+1) -\zeta_\nu(n)\cdots\zeta_\nu(n-m+2)\right) \\
&= \frac{1}{q^{m-1}-1}\zeta_\nu(n-1)\cdots\zeta_\nu(n-m+2)\left(q^{m-1}\zeta_\nu(n-m+1)-\zeta_\nu(n)\right),
\end{align*}
but one can easily verify that
\begin{equation*}
\frac{q^{m-1}\zeta_\nu(n-m+1)-\zeta_\nu(n)}{q^{m-1}-1} = \zeta_\nu(n)\zeta_\nu(n-m+1),
\end{equation*}
proving \eqref{eq:zeta_local}, as desired.

Thus it remains to prove \eqref{eq:Q} above. We again argue by induction on $m$. There is nothing to prove for the base case $m=1$. For the general case, observe that each and every coset appearing in $\mathcal T(\nu^k)$ has a representative of the form
\begin{equation*}
\begin{pmatrix}
\pi_\nu^{a_1} & *                            & *         & \cdots & * \\
                      & \pi_\nu^{a_2-a_1} & *         &            & * \\
                      &                              & \ddots &            & \vdots  \\
                     &                         &     & \pi_\nu^{a_{m-1}-a_{m-2}} & * \\
                      &                              &             & & \pi_\nu^{a_{m}-a_{m-1}}
\end{pmatrix},
\end{equation*}
where $0 = a_0 \leq a_1 \leq \ldots \leq a_m = k$, and each entry to the right of $\pi_\nu^{a_{i}-a_{i-1}}$ is chosen uniquely mod $\pi_\nu^{a_{i}-a_{i-1}}$. Since
\begin{equation*}
\prod_{i=1}^{m-1} q^{(m-i)(a_i-a_{i-1})} = q^{a_1 + \ldots + a_{m-1}},
\end{equation*}
it follows that
\begin{equation*}
Q^{(m)}_\nu(k) = \sum_{0 \leq a_1 \leq \ldots \leq a_{m-1} \leq k} q^{a_1 + \ldots + a_{m-1}},
\end{equation*}
which in turn implies
\begin{equation*}
Q^{(m)}_\nu(k) = \sum_{a=0}^k q^aQ^{(m-1)}_\nu(a)
\end{equation*}
(think $a = a_{m-1}$). By the summation by parts, we obtain
\begin{equation*}
Q^{(m)}_\nu(k) = Q^{(m-1)}_\nu(k)\frac{q^{k+1}-1}{q-1} - \sum_{a=1}^k \frac{q^a-1}{q-1}\left(Q^{(m-1)}_\nu(a)-Q^{(m-1)}_\nu(a-1)\right).
\end{equation*}
The induction hypothesis implies
\begin{align*}
&Q^{(m-1)}_\nu(a)-Q^{(m-1)}_\nu(a-1) \\
&=\frac{q^{a+1}-1}{q-1} \cdots \frac{q^{a+m-2}-1}{q^{m-2}-1} - \frac{q^{a}-1}{q-1} \cdots \frac{q^{a+m-3}-1}{q^{m-2}-1} \\
&=q^aQ^{(m-2)}_\nu(a) = \frac{q^a(q^{m-2}-1)}{q^a-1}Q^{(m-1)}_\nu(a-1).
\end{align*}
Therefore
\begin{align*}
Q^{(m)}_\nu(k) &= Q^{(m-1)}_\nu(k)\frac{q^{k+1}-1}{q-1} - \frac{q^{m-2}-1}{q-1}\sum_{a=1}^k q^aQ^{(m-1)}_\nu(a-1) \\
&= Q^{(m-1)}_\nu(k)\frac{q^{k+1}-1}{q-1} - Q^{(m)}_\nu(k-1)\frac{q^{m-1}-q}{q-1} \\
&= Q^{(m-1)}_\nu(k)\frac{q^{k+1}-1}{q-1} - \left(Q^{(m)}_\nu(k) - q^{k}Q^{(m-1)}_\nu(k)\right)\frac{q^{m-1}-q}{q-1}.
\end{align*}
From this, we obtain
\begin{align*}
&\left(1+\frac{q^{m-1}-q}{q-1}\right)Q^{(m)}_\nu(k) = \left(\frac{q^{k+1}-1}{q-1}+\frac{q^k(q^{m-1}-q)}{q-1}\right)Q^{(m-1)}_\nu(k) \\
&\Rightarrow \frac{q^{m-1}-1}{q-1}Q^{(m)}_\nu(k) = \frac{q^{m+k-1}-1}{q-1}Q^{(m-1)}_\nu(k) \\
&\Rightarrow Q^{(m)}_\nu(k) = \frac{q^{m+k-1}-1}{q^{m-1}-1}Q^{(m-1)}_\nu(k).
\end{align*}
This completes the proof of \eqref{eq:Q}.
\end{proof}

\subsection{Rogers integral formula, the higher level case}

As an application of Theorem \ref{thm:prim}, we prove the following result. As will be shown later in this paper, the full adelic Rogers formula, Theorem \ref{thm:intro_main}, will follow as its corollary.


\begin{theorem}\label{thm:cong}
Let $1 \leq m < n$, and let $f: (\A_F^n)^m \rightarrow \R$ be of form $f = f_\f f_\infty$ such that $f_\infty$ is a bounded and compactly supported Riemann integrable function on $(\A_\infty^n)^m$, and $f_\f$ is the characteristic function of $v + \eta\prod_{\nu \in \f} (\cO_\nu^n)^m \subseteq (\A_\f^n)^m$ for some $v \in (\A_\f^n)^m$ and $\eta \in \A_\f^*$. Then
\begin{equation*}
\int_{\mathcal X_n} \sum_{X \in (F^n)^m \atop \mathrm{indep.}} f(Xg) d\mu_n = \int_{(\A_F^n)^m} f(X) d\alpha_F^{nm}.
\end{equation*}
\end{theorem}

\begin{proof}
Let $K_\nu = \GL(n, \cO_\nu)$, and let $K_\nu(e_\nu)$ be the kernel of the projection $K_\nu \rightarrow \GL(n, \cO_\nu/\pi_\nu^{e_\nu}\cO_\nu)$ for an integer $e_\nu \geq 1$; we put $K_\nu(0) = K_\nu$. Also, for the sequence $e = (e_\nu)_{\nu \in \f}$ where $e_\nu \in \Z_{\geq 0}$, all but finitely many of which are $0$, write $K(e) := \prod_{\nu \in \f} K_\nu(e_\nu)$.

Our idea for the proof is to choose an appropriate $e$ such that $f(X) = f(Xk)$ for $k \in K(e)$, and consider
\begin{equation}\label{eq:cong}
\frac{1}{|K : K(e)|} \sum_{k \in K/K(e)} \sum_{X \in (F^n)^m \atop \mathrm{indep.}} f(Xgk)
\end{equation}
instead, noting that this and $\sum_{X \in (F^n)^m \atop \mathrm{indep.}} f(Xg)$ yields the same value when integrated over $\mathcal X_n$. We will show that \eqref{eq:cong} is equal to
\begin{equation}\label{eq:cong2}
\frac{1}{|K : K(e)|} \sum_{X \in (F^n)^m \atop \mathrm{indep.}} \tilde{f}(Xg),
\end{equation}
where $\tilde{f} = \tilde f_\f \tilde f_\infty$ is a certain modification of $f$ that will be described more carefully below. The key point is that $\tilde f_\f$ will be the characteristic function of a set ``centered at'' $0 \in (\A_\f^n)^m$, to which Theorem \ref{thm:meanhecke} becomes applicable, yielding us Theorem \ref{thm:cong}.

Let us fill in the details of the argument outlined above. The sequence $e = (e_\nu)_{\nu \in \f}$ is determined by the condition
\begin{equation*}
e_\nu = \min\{e_\nu \in \Z_{\geq 0} : \pi_\nu^{e_\nu}v_\nu \in \eta_\nu(\cO_\nu^n)^m\}.
\end{equation*}
Indeed $e_\nu = 0$ for almost all $\nu$. It can also be easily checked that $f_\f(Xk) = f_\f(X)$ for $k \in K(e)$.

$\tilde f$ is defined as follows: $\tilde f_\infty = f_\infty$, $\tilde f_\nu = f_\nu$ for $\nu$ with $e_\nu = 0$, and for those $\nu$ such that $e_\nu > 0$, $\tilde f_\nu$ is $\sigma_\nu$ times the characteristic function of $S_\nu := (v_\nu + \eta_\nu\cO_\nu^{nm})K_\nu \subseteq (F_\nu^n)^m$, and $\sigma_\nu$ is the constant satisfying
\begin{equation}\label{eq:sigma}
\frac{\sigma_\nu\alpha_\nu^{nm}(S_\nu)}{|K_\nu : K_\nu(e_\nu)|} = \alpha_\nu^{nm}(v_\nu + \eta_\nu\cO_\nu^{nm}).
\end{equation}
We defer the proof of the equality of \eqref{eq:cong} and \eqref{eq:cong2} to the next section, and contiue with the proof of Theorem \ref{thm:cong}.

We need the following description of $S_\nu$, in order to apply Theorem \ref{thm:meanhecke} later.
Recall that, by the theory of the Smith normal form, we can write
\begin{equation*}
v_\nu = \gamma \cdot \diag(\pi_\nu^{a_1}, \ldots, \pi_\nu^{a_m}) \cdot P
\end{equation*}
for some $\gamma \in K_\nu$ and $P \in \mathrm{Mat}_{m \times n}(\cO_\nu)$ primitive, and $a_1 \leq \ldots \leq a_m$. Suppose $m'$ is the greatest index for which $b:=\ord_\nu \eta_\nu > a_{m'}$. Then, with $a := \diag(\pi_\nu^{a_1}, \ldots, \pi_\nu^{a_{m'}}, \pi_\nu^b, \dots, \pi_\nu^b)$, we have
\begin{align*}
S_\nu = \gamma(aP+\eta_\nu\cO^{nm})K_\nu = \gamma a P_{m'},
\end{align*}
where $P_{m'}$ is the set of all elements in $\mathrm{Mat}_{m \times n}(\cO_\nu)$ whose first $m'$ rows form a primitive $m' \times n$ matrix.

By Proposition \ref{prop:indprim}, we have
\begin{equation*}
\tilde{f}_\nu(X) = \sigma_\nu \mathbf 1_{P_{m'}}(a^{-1}\gamma^{-1}X) = \sigma_\nu \left( \sum_{i=0}^{m'}(-1)^iN(\nu)^{i(i-1)/2}{\mathcal T'}^{(i)}(\nu) \right) \mathbf 1_{\cO_\nu^{nm}}(a^{-1}\gamma^{-1}X),
\end{equation*}
where ${\mathcal T'}$ is $\mathcal T$ in dimension $m'$; more precisely,
\begin{equation*}
{\mathcal T'}^{(i)}f(X) = \frac{1}{\omega_{\nu, m'}(K'_\nu \alpha'_\nu K'_\nu)} \int_{K'_\nu \alpha'_\nu K'_\nu} f\left(\begin{pmatrix} \gamma^{-1} & \\  & \mathrm{Id}_{m-m'} \end{pmatrix} X\right)d\omega_{\nu,m'},
\end{equation*}
where $K'_\nu = \GL(m', \cO_\nu)$ and $\alpha'_\nu = \mathrm{diag}(1, \ldots, 1, \nu, \ldots, \nu)$  with $m'-i$ $1$'s and $i$ $\nu$'s. Therefore we may apply Theorem \ref{thm:meanhecke} to \eqref{eq:cong2}, and \eqref{eq:sigma} ensures that it has the desired measure.
\end{proof}

\subsection{Equivalence of \eqref{eq:cong} and \eqref{eq:cong2}}

We continue with the notations from the previous section. From the definition of $S_\nu$, it is clear that
\begin{equation*}
\sum_{k \in K_\nu / K_\nu(e_\nu)} f_\nu(Xk) = \sigma_\nu \mathbf 1_{S_\nu}(X),
\end{equation*}
where $\sigma_\nu$ is the order of the stabilizer of $v_\nu + \eta_\nu\cO_\nu^{nm}$ in $K_\nu(e_\nu) \backslash K_\nu$. Since
\begin{equation*}
v_\nu + \eta_\nu\cO_\nu^{nm} = \gamma a (P + \diag(\pi_\nu^{l_1}, \ldots, \pi_\nu^{l_m})\cO^{nm})
\end{equation*}
where $l_1 = b-a_1, \ldots, l_{m'} = b-a_{m'}, l_{m'+1} = \ldots = l_m = 0$ and $P \in P_{m'}$, $\sigma_\nu$ is also the order of the stabilizer of
$P + \diag(\pi_\nu^{l_1}, \ldots, \pi_\nu^{l_m})\cO^{nm}.$
To show that \eqref{eq:cong} and \eqref{eq:cong2} are equal, all we need to show is that this $\sigma_\nu$ satisfies \eqref{eq:sigma}. This will be done through a series of lemmas below. We write $q = N(\nu)$ for brevity in what follows.

\begin{lemma}\label{lemma:cong1}
The order of the group $\GL(n,\cO_\nu / \pi_\nu^l\cO_\nu)$ is $q^{(l-1)n^2}(q^n-1)(q^n-q) \cdots (q^n-q^{n-1})$.
\end{lemma}
\begin{proof}
The case $l=1$ is well-known: $|\GL(n,\cO_\nu / \pi_\nu\cO_\nu)| = (q^n-1)(q^n-q) \cdots (q^n-q^{n-1})$. For general $l$, consider the surjection $\GL(n,\cO_\nu / \pi_\nu^l\cO_\nu) \twoheadrightarrow \GL(n,\cO_\nu / \pi_\nu\cO_\nu)$ induced by reduction modulo $\pi_\nu$. This map has kernel $\mathrm{Id} + \pi_\nu\mathrm{Mat}_{n \times n}(\cO_\nu/\pi_\nu^l\cO_\nu)$, which has order $q^{(l-1)n^2}$.
\end{proof}

\begin{lemma}\label{lemma:cong2}
Let $1 \leq m' \leq m < n$, and let $P \in P_{m'}$. Also let $l=l_1 \geq l_2 \geq \ldots \geq l_{m'} > l_{m'+1} = \ldots = l_m = 0$. Then the order of the stabilizer of $P + \diag(\pi_\nu^{l_1}, \ldots, \pi_\nu^{l_m})\cO_\nu^{nm}$ (mod $\pi_\nu^l$) in $\GL(n,\cO_\nu / \pi_\nu^l\cO_\nu)$ is 
\begin{equation*}
q^{lm'(n-m')}|\GL(n-m',\cO_\nu / \pi_\nu^l\cO_\nu)| \prod_{i=2}^{m'}q^{n(l_1-l_i)}.
\end{equation*}
\end{lemma}
\begin{proof}
Without loss of generality, we may assume that $P$ is the matrix
\begin{equation*}
\begin{pmatrix}
1 &  &  &  & & \\
   & \ddots &  & &  & \\
   &            & 1 & &  &  \\
\end{pmatrix}
\end{equation*}
whose $(i,i)$-entries are $1$ for $i=1, \ldots, m$, and the rest of the entries are zero. Then the stabilizer consists of matrices of the form $A+B$, where $A$ is of the form
\begin{equation*}
\begin{pmatrix}
\mathrm{Id}_{m' \times m'} & \\
\mathrm{Mat}_{(n-m') \times m'}(\cO_\nu/\pi_\nu^l\cO_\nu) & \GL(n-m',\cO_\nu / \pi_\nu^l\cO_\nu)
\end{pmatrix},
\end{equation*}
and $B$ is an $n \times n$ matrix whose $i$-th row is an element of $(\pi_\nu^{l_i}\cO_\nu / \pi_\nu^l\cO_\nu)^n$ for $i=2, \ldots, m'$, and all the remaining entries are zero. It is clear that the set of such matrices has the said order.
\end{proof}

\begin{lemma}\label{lemma:cong3}
Continue with the notations and assumptions above. Then we have
\begin{equation} \label{eq:vci}
\frac{\left|\mathrm{stab}(P + \diag(\pi_\nu^{l_1}, \ldots, \pi_\nu^{l_m})\cO_\nu^{nm})\right|}{\left|\GL(n,\cO_\nu / \pi_\nu^l\cO_\nu)\right|}\zeta_\nu(n)^{-1} \cdots \zeta_\nu(n-m'+1)^{-1} = q^{-n(l_1 + \ldots + l_{m'})}.
\end{equation}
\end{lemma}


\begin{proof}
By the previous two lemmas,
\begin{align*}
&\frac{\left|\mathrm{stab}(P + \diag(\pi_\nu^{l_1}, \ldots, \pi_\nu^{l_m})\cO_\nu^{nm})\right|}{\left|\GL(n,\cO_\nu / \pi_\nu^l\cO_\nu)\right|} \\
&= \frac{q^{lm'(n-m')+(l-1)(n-m')^2}\prod_{i=2}^{m'}q^{n(l_1-l_i)}(q^{n-m'}-1) \cdots (q^{n-m'} - q^{n-m'-1})}{q^{(l-1)n^2}(q^n-1) \cdots (q^n-q^{n-1})} \\
&= q^{lm'(n-m')+(l-1)(n-m')^2-(l-1)n^2}\prod_{i=2}^{m'}q^{n(l_1-l_i)} \cdot q^{-m'(n-m')-nm'}\zeta_\nu(n) \cdots \zeta_\nu(n-m'+1).
\end{align*}
Thus the left-hand side of \eqref{eq:vci} equals
\begin{align*}
&q^{lm'(n-m')+(l-1)(n-m')^2-(l-1)n^2}\prod_{i=2}^{m'}q^{n(l_1-l_i)} \cdot q^{-m'(n-m')-nm'} \\
&=q^{-n(l_1 + \ldots + l_{m'})}q^{m'nl} \cdot q^{lm'(n-m')+(l-1)(n-m')^2-(l-1)n^2-m'(n-m')-nm'} \\
&=q^{-n(l_1 + \ldots + l_{m'})}q^{lm'(2n-m') +(l-1)(-2nm'+{m'}^2) -2nm' + {m'}^2} \\
&=q^{-n(l_1 + \ldots + l_{m'})},
\end{align*}
as desired.
\end{proof}

Recalling that $|K_\nu: K_\nu(e_\nu)| = |\GL(n,\cO_\nu / \pi_\nu^{b-a_1}\cO_\nu)|$ and $\alpha_\nu^{nm}(P_{m'}) = \zeta_\nu(n)^{-1} \cdots \zeta_\nu(n-m'+1)^{-1}$, Lemma \ref{lemma:cong3} implies
\begin{align*}
\frac{\sigma_\nu \alpha_\nu^{nm}(S_\nu)}{|K_\nu : K_\nu(e_\nu)|} =q^{-n(l_1+\ldots l_{m'})}\frac{\alpha_\nu^{nm}(S_\nu)}{\alpha_\nu^{nm}(P_{m'})} = q^{-nmb} = \alpha_\nu^{nm}(v_\nu + \eta_\nu\cO_\nu^{nm}),
\end{align*}
as desired.

\subsection{Proof of Theorem \ref{thm:intro_main}}

We will prove, for a Borel integrable $f : (\A_F^n)^k \rightarrow \R$,
\begin{equation} \label{eq:last}
\int_{\mathcal X_n} \sum_{X \in (F^n)^k \atop \mathrm{indep.}} f(Xg) d\mu = \int_{(\A_F^n)^k} f(X) d\alpha_F^{nk}
\end{equation}
by a series of reductions to smaller families of functions, until we reach those functions covered by Theorem \ref{thm:cong}. 

First of all, it is clear that we may assume $f$ is nonnegative, since the general $f$ is a difference of nonnegative functions. It is also clear that we may assume $f$ is compactly supported and bounded: choose a countable sequence of compact measurable subsets $S_1 \subseteq S_2 \subseteq \ldots$ of $(\A_F^n)^k$ such that $\bigcup_N S_N = (\A_F^n)^k$, and define $f_N(X) = \mathbf{1}_{S_N}(X) \cdot \min(f(X),N)$. Then by the monotone convergence theorem, if \eqref{eq:last} holds for all $f_N$, then it holds for $f$ too.

A compactly supported and bounded measurable function $f$ can in turn be approximated pointwise by compactly supported simple functions. Both are dominated by a Riemann integrable function e.g. $\|f\|_\infty$ times the characteristic function of any open set containing $\mathrm{supp}\,f$, hence the dominated convergence theorem applies, implying that \eqref{eq:last} holds for $f$ if it holds for the simple functions. Such a simple function, in turn, can be approximated pointwise by a finite linear combinations of the characteristic functions of the sets of the form $S_\f \times S_\infty$, where $S_\f = v + \eta\prod_{\nu \mid \f}\cO_\nu^{nk} \subseteq (\A_\f^n)^k$ and $S_\infty \subseteq (\A_\infty^n)^k$ is a bounded measurable set.

In summary, we reduced our task to showing \eqref{eq:last} for the functions $f$ of the form $f = \mathbf{1}_{S_\f}\mathbf{1}_{S_\infty}$. The only remaining obstruction to applying Theorem \ref{thm:cong} to this $f$ is that $\mathbf{1}_{S_\infty}$ may not be Riemann integrable. We choose a sequence of bounded, compactly supported continuous --- therefore Riemann integrable --- functions $f_{\infty,N}$ on $\A_\infty^{nk}$ that converges to $\mathbf{1}_{S_\infty}$ in the $L^1$-norm. Accordingly we write $f_N = \mathbf{1}_{S_\f}f_{\infty,N}$, which of course converges to $f$ in the $L^1$-norm. The pointwise limit of $f_N$ and $f$ can disagree at most on a null set, say $Z$. However, the set
\begin{equation*}
\{ g \in G_n : F^{nm}g \cap Z \neq \phi \}
\end{equation*}
has $\mu_n$-measure zero, and hence the convergence
\begin{equation*}
\sum_{X \in (F^n)^k \atop \mathrm{no\, zero\, rows}} f_N(Xg) \rightarrow \sum_{X \in (F^n)^k \atop \mathrm{no\, zero\, rows}} f(Xg)
\end{equation*}
holds for $\mu_n$-almost every $g$. Hence we can cite the dominated convergence theorem and conclude that \eqref{eq:last} holds for $f$. This completes the proof of the first formula claimed by Theorem \ref{thm:intro_main}.

The other formula of Theorem \ref{thm:intro_main} follows from the previous and the equality
\begin{equation*}
\sum_{X \in (F^n)^k \atop \mathrm{no\, zero\, rows}} f(X) = \sum_{m=1}^{k}\sum_D\sum_{X \in (F^n)^m \atop \mathrm{indep.}} f(D^\mathrm{tr} X),
\end{equation*}
where $D$ runs over all $m \times k$ row-reduced echelon forms over $F$ of rank $m$. More precisely, suppose $f \geq 0$ without loss of generality, and consider the sum
\begin{equation*}
\sum_{m=1}^{k}\sum_{D_N} f(D_N^\mathrm{tr} X),
\end{equation*}
where $D_N$ runs over all $m \times k$ row-reduced echelon forms over $F$ of rank $m$, each of whose nonzero entries, say $z$, satisfies (i) $|\ord_\nu z_\nu| \leq N/N(\nu)$ (ii) $\|\Log(z)\| \leq N$. Then this expression is a finite sum that is pointwise monotonically increasing in $N$. Of course, the same can be said of
\begin{equation*}
\sum_{m=1}^{k}\sum_{D_N}\sum_{X \in (F^n)^m \atop \mathrm{indep.}} f(D_N^\mathrm{tr} Xg).
\end{equation*}
Hence we can invoke the monotone convergence theorem --- twice, for $\A_F$ and $\mathcal X_n$ respectively --- to complete the proof.

\section{Second moment estimates}

\subsection{Proof of Theorem \ref{thm:intro_second}}

We recall the setting: $n \geq 3$, and $f : \A_F^n \rightarrow \R$ is a nonnegative function of the form $f = f_\f f_\infty$, where $f_\f$ is the characteristic function of an integrable set $\prod_{\nu \in \f} A_\nu \subseteq \A_\f^n$, and $f_\infty$ satisfies, for any $\gamma \in F^*$ and a constant $C$, a bound of the form
\begin{equation*}
\int_{\A_\infty^n} f_\infty(x)f_\infty(\gamma x)d\alpha_\infty^n \leq C\alpha^n_\infty(f_\infty)\min(1, \min_{\sigma \mid \infty} \|\gamma\|_\sigma^{-1})^n.
\end{equation*}

A few lemmas are in order.

\begin{lemma} \label{lemma:unitcount}
For $\gamma \in F^*$ and $A \subseteq \R$, let
\begin{equation*}
M(\gamma, A) = \frac{1}{w_F}\mbox{(\# $u \in \cO_F^*$ such that the largest coordinate of $\mathrm{Log}(\gamma u)$ is in $A$)}.
\end{equation*}
Then, if $A = (-\infty, k]$ and $(r+1)k - \log |N\gamma| \geq 0$,
\begin{equation*}
M(\gamma, A) = \frac{1}{R_F} \cdot \frac{\sqrt{r+1}}{r!}\left((r+1)k - \log |N\gamma|\right)^r + O_F(\left((r+1)k - \log |N\gamma|\right)^{r-1}).
\end{equation*}
If $(r+1)k - \log |N\gamma| < 0$ then $M(\gamma, A) = 0$. In particular, for general $A \subseteq \R$ bounded from above, $M(\gamma, A)$ is finite.
\end{lemma}
\begin{proof}
$M(\gamma,A)$ is precisely the number of points of the unit lattice of $F$, translated by $\mathrm{Log}\ \gamma$, whose maximum coordinate is at most $k$. If $(r+1)k < \log |N\gamma|$, then there exists no such points. If $(r+1)k \geq \log |N\gamma|$, this is the number of the points of the translate of the unit lattice inside the simplex formed by connecting the vertices $(k, \ldots, k, -rk + \log |N\gamma|), (k, \ldots, k, -rk + \log |N\gamma|, k),$ and so on, which is a regular $r$-simplex of side length $\sqrt{2}((r+1)k - \log |N\gamma|)$. The proof follows from the volume formula of a regular simplex and a standard lattice-point counting estimate.
\end{proof}

\begin{lemma} \label{lemma:2termunitest}
Fix $\gamma \in F^*$. Then
\begin{equation*}
\sum_{u \in \cO^*_F} \int_{\A_\infty^n} f_\infty(x)f_\infty(\gamma ux) d\alpha_\infty^n = O_F( (1 + |\log |N\gamma||^r) C\alpha_\infty^n(f_\infty)).
\end{equation*}

\end{lemma}
\begin{proof}
By assumption, for each $u \in \mathcal{O}^*_F$, the integral inside the sum is bounded by the smaller of
\begin{equation*}
C\alpha_\infty^n(f_\infty) \mbox{ or } C a^n \alpha_\infty^n(f_\infty),
\end{equation*}
where $a = \min_\sigma \|\gamma u\|_\sigma^{-1}$.
If $a \geq 1$, or equivalently, if the largest coordinate of $\mathrm{Log}(\gamma u)$ is at most $0$, we bound the integral by the former. For an integer $k \geq 0$, if $a \in [e^{-(k+1)}, e^{-k})$, or equivalently, if the largest coordinate of $\mathrm{Log}(\gamma u)$ is in $(k, k + 1]$, we bound the integral by $Ce^{-nk}\alpha_\infty^n(f_\infty)$. Thus the left-hand side of the claimed equality is at most $w_FC\alpha_\infty^n(f_\infty)$ times
\begin{equation*}
M(\gamma, (-\infty, 0]) + \sum_{k = 0}^\infty M(\gamma, (k, k + 1]) e^{-nk}.
\end{equation*}
From Lemma \ref{lemma:unitcount} it follows that
\begin{equation*}
M(\gamma, (-\infty, 0]) = O_F(1 + |\log |N\gamma||^r),
\end{equation*}
and also that
\begin{equation*}
M(\gamma, (x, x+1]) = O_F((1 + |\log |N\gamma||^r)x^{r-1}).
\end{equation*}
But then
\begin{equation*}
\sum_{k=0}^\infty k^{r-1}e^{-nk} = O_F(1),
\end{equation*}
which completes the proof.
\end{proof}

Now Theorem \ref{thm:intro_main} implies that
\begin{equation*}
\int_{{\mathcal X}_n} \left( \sum_{x \in F^n \backslash \{0\}} f(xg) \right)^2 d\mu = \int_{\A^{2n}} f(x_1) f(x_2) d\alpha^n(x_1)d\alpha^n(x_2) + \sum_{c \in F^*} \int_{\A^n} f(x)f(cx) d\alpha^n(x).
\end{equation*}
On the right-hand side, the former integral is simply $\alpha_F(f)^2$. It remains to show that the latter sum of integrals is $O_F(C\alpha_F(f))$. Let us rewrite it as
\begin{equation*}
2 \sum_{c \in F^* \atop N(c) \leq 1} \int_{\A^n} f(x)f(cx) d\alpha^n.
\end{equation*}
Thanks to the exact sequence
\begin{equation*}
0 \rightarrow \cO^*_F \rightarrow F^* \rightarrow \mathrm{Prin}_F \rightarrow 0
\end{equation*}
where $\mathrm{Prin}_F$ denotes the set of all principal fractional ideals of $F$, we can rewrite the above as
\begin{equation} \label{eq:oink}
2|\Delta_F|^{-\frac{n}{2}}\sum_{I \in \mathrm{Prin}_F \atop N(I) \leq 1} \sum_{u \in \mathcal{O}^*_F} \int_{\mathbb{A}^n_{\f}} f_\f(x_\f)f_\f((\gamma(I) u x)_\f) d\alpha^n_\f \int_{\mathbb{A}^n_\infty} f_\infty(x_\infty) f_\infty((\gamma(I) ux)_\infty) d\alpha^n_\infty,
\end{equation}
where $\gamma(I) \in F^*$ is a choice of a generator of $I$. Since two different choices of $\gamma(I)$ differ by an element of $\cO_F^*$, \eqref{eq:oink} is well-defined.

Any $I \in \mathrm{Prin}_F$ has the unique factorization $I = pq^{-1}$, where $p,q$ are coprime integral ideals of $F$ of the same ideal class. Therefore, \eqref{eq:oink} is, up to a constant factor depending on $F$ only,
\begin{equation}\label{eq:2ndmtp1}
\sum_q \sum_p \sum_{u \in \mathcal{O}^*_F} \int_{\mathbb{A}^n_{\f}} f_\f(x_\f)f_\f((\gamma u x)_\f) d\alpha^n_\f \int_{\mathbb{A}^n_\infty} f_\infty(x_\infty) f_\infty((\gamma ux)_\infty) d\alpha^n_\infty,
\end{equation}
where $q$ runs over all integral ideals of $F$, and $p$ runs over the integral ideals coprime to $q$ and in the same ideal class as $q$, and $Np \leq Nq$; $\gamma = \gamma(pq^{-1})$ here.

To handle the integral over $\A_\f^n$, we observe that, for any $\delta_\nu \in F_\nu^*$,
\begin{align}
\int_{F_\nu^n} f_\nu(x_\nu)f_\nu(\delta_\nu x_\nu) d\alpha_\nu^n &= \alpha_\nu^n(A_\nu \cap \delta_\nu^{-1} A_\nu) \notag \\
& \leq \min(\alpha^n_\nu(A_\nu), \|\delta_\nu\|^{-n}_\nu \alpha^n_\nu(A_\nu)). \label{eq:oink2}
\end{align}

Hence, \eqref{eq:oink} is bounded by a constant term times
\begin{equation*}
\sum_q \frac{1}{(Nq)^n} \sum_p \sum_{u \in \mathcal{O}^*_F} \alpha^n_\f(f_\f)\int_{\mathbb{A}^n_\infty} f_\infty(x_\infty) f_\infty((\gamma ux)_\infty) d\alpha^n_\infty,
\end{equation*}

We now apply Lemma \ref{lemma:2termunitest} to bound this by
\begin{equation*}
\sum_q \frac{1}{(Nq)^n} \sum_p O_F\left((1 + |\log |N\gamma||^r)C\alpha^n_F(f) \right).
\end{equation*}

It remains to show that
\begin{equation*}
\sum_q \frac{1}{(Nq)^n} \sum_p (1 + |\log |N\gamma||^r) = O_F(1).
\end{equation*}
By the Dedekind-Weber theorem, the left-hand side is bounded by a constant (depending on $F$) times
\begin{equation*}
\sum_q \frac{1}{(Nq)^{n-1}}(1 + \log^r Nq) \ll \zeta_F(n-3/2),
\end{equation*}
as desired, completing the proof of Theorem \ref{thm:intro_second}

\subsection{The case of balls and annuli}

In what follows, $V_n$ denotes the volume of the unit ball in $\R^n$. In accordance with the metric we assigned on $\A_\infty$, the (closed) annulus at origin of radii $R_1 < R_2$ in $\A_\infty^n$ is the set of $x \in \A_\infty^n$ satisfying
\begin{equation*}
R_1^2 \leq \sum_{\sigma\ \mathrm{real}} |x_{\sigma}|^2 + 2\sum_{\sigma\ \mathrm{cplx}} |x_\sigma|^2 \leq R_2^2,
\end{equation*}
where $|\cdot|$ is the standard Euclidean metric on $\R^n$ or $\C^n$, as appropriate. $R_1 = 0$ is permitted, in which case the annulus is in fact a ball of radius $R_2$.
With respect to the measure $\alpha_\infty^n$, a ball of radius $R$ has volume $V_{nd}R^{nd}$, same as the volume of the ball of radius $R$ in $\R^{nd}$ with respect to the standard Euclidean metric.

The purpose of this section is to prove the following lemma, which shows that Theorem \ref{thm:intro_second} is applicable when $f_\infty$ is the characteristic function of an annulus at origin.


\begin{lemma} \label{lemma:2termest}
Let $n \geq 2$, and let $f_\infty$ be the characteristic function of an annulus on $\mathbb{A}_\infty^n$ centered at origin.
Then for $\gamma \in F^*$,
\begin{equation*}
\int_{\mathbb{A}_\infty^n} f_\infty(x)f_\infty(\gamma x) d\alpha_\infty^n \leq \alpha_\infty^n(f_\infty) \min\left(1, de\min_{\sigma \mid \infty} \|\gamma\|_\sigma^{-1} \right)^n.
\end{equation*}

\end{lemma}

\begin{proof}

The case $d = 1$ (i.e. $F = \Q$) is trivial, so let us assume $d \geq 2$. Choose a place $\sigma \mid \infty$ minimizing $\|\gamma\|_\sigma^{-1}$, and write $a = \|\gamma\|_\sigma^{-1}$. We assume $a < 1$, since otherwise again the lemma is trivial.

Write $B:= \mathrm{supp}\,f_\infty$, an annulus at origin of radii $R_1 < R_2$. Let $P_\sigma: (F \otimes \R)^n \rightarrow F_\sigma^n$ be the orthogonal projection onto the $\sigma$-coordinate. Also, for $z \in F_\sigma^n$, write $P^\perp_\sigma(B, z) = \{y \in \prod_{l \mid \infty \atop l \neq \sigma} F_l^n: (y, z) \in B\}$ i.e. the ``slice'' of $B$ at $z$. Then we have
\begin{equation} \label{eq:tetete}
\int_{\mathbb{A}_\infty^n} f_\infty(x)f_\infty(\gamma x) d\alpha_\infty^n \leq \int_{P_\sigma(\gamma^{-1}B)} \left( \int_{P^\perp_\sigma(B, z)} \prod_{l \mid \infty \atop l \neq \sigma} d\alpha_l^n \right) d\alpha_{\sigma}^n.
\end{equation}

Assume first that $\sigma$ is real.  Then $P_\sigma(\gamma^{-1}B)$ is the ball in $\R^n$ at origin of radius $aR_2$. On the other hand, if $z$ has length $r$,
\begin{equation*}
\int_{P^\perp_\sigma(B, z)} \prod_{l \mid \infty \atop l \neq \sigma} d\alpha_l^n = V_{n(d-1)}\left((R_2^2 - r^2)^{n(d-1)/2} - \max(0,R_1^2 - r^2)^{n(d-1)/2}\right).
\end{equation*}
From the fact that
\begin{equation*}
\frac{d}{dr}(R^2 - r^2)^{n(d-1)/2} = -rn(d-1)(R^2 - r^2)^{n(d-1)/2 - 1} \leq 0
\end{equation*}
for $r \in (0,R)$, one can deduce that
\begin{equation*}
\int_{P^\perp_\sigma(B, z)} \prod_{l \mid \infty \atop l \neq \sigma} d\alpha_l^n \leq V_{n(d-1)}(R_2^{n(d-1)} - R_1^{n(d-1)}).
\end{equation*}
Thus the right-hand side of \eqref{eq:tetete} is bounded by
\begin{align*}
& V_nV_{n(d-1)}\int_{0}^{a R_2} (R_2^{n(d-1)} - R_1^{n(d-1)}) r^{n-1} dr \\
& \leq \frac{V_nV_{n(d-1)}}{n} a^n (R_2^{n(d-1)} - R_1^{n(d-1)}) R_2^n \\
& \leq \frac{V_nV_{n(d-1)}}{n} a^n (R_2^{nd} - R_1^{nd}) \\
& \leq \frac{V_nV_{n(d-1)}}{nV_{nd}} a^n \alpha_\infty^n(f_\infty).
\end{align*}

If $\sigma$ is complex, then $P_\sigma(\gamma^{-1}B)$ is the ball in $\C^n \cong \R^{2n}$ at origin of radius $aR_2$.
Similarly as in the real case, we obtain the bound
\begin{equation*}
\frac{V_{2n}V_{n(d-2)}}{2nV_{nd}} a^n \alpha_\infty^n(f_\infty)
\end{equation*}
for both cases. By Stirling's formula, one can compute that
\begin{equation*}
\max\left(\frac{V_{n}V_{n(d-1)}}{nV_{nd}}, \frac{V_{2n}V_{n(d-2)}}{2nV_{nd}}\right) < (de)^n.
\end{equation*}

Collecting all our computations so far, we can bound \eqref{eq:tetete} by
\begin{equation*}
(ade)^n\alpha_\infty^n(f_\infty),
\end{equation*}
which is the desired bound.
\end{proof}

\subsection{Estimate for $\mathbb P^{n-1}(F)$}


In this section, we prove Theorem \ref{thm:intro_proj}. Recall that, for $g \in G_n$ we defined $P_B(g)$ to be the number of points on $\mathbb P^{n-1}(F)$ whose twisted height with respect to $g$ is at most $B$.
We will construct the function $f : \A_F^n \rightarrow \R$ such that
\begin{equation*}
P_B(g) = \sum_{x \in F^n \backslash \{0\}} f(B_\infty^{-\frac{1}{d}}xg).
\end{equation*}
where
\begin{equation*}
B_\infty^{-\frac{1}{d}} = (\underbrace{1, \ldots, 1}_{\mbox{\tiny finite places}}, \underbrace{B^{-\frac{1}{d}}, \ldots, B^{-\frac{1}{d}}}_{\mbox{\tiny infinite places}}).
\end{equation*}
Due to the way the height is defined, it suffices to consider the case $B = 1$ only.

Let $h=h_F$ be the class number of $F$, and choose a set of prime ideals $\{I_1, \ldots, I_h\}$ that also serve as the set of the representatives of the ideal classes of $F$. For each $1 \leq i \leq h$, let $a_i$ be the norm $1$ idele such that $(a_i)_\nu = \pi_\nu^{\ord_\nu(I_i)}$ for $\nu \in \f$, and $(a_i)_\sigma = \|N(I_i)^{\frac{1}{d}}\|_\sigma$ for $\sigma \mid \infty$. In addition, let
\begin{equation*}
H = \left\{(x_1, \ldots, x_{r_1+r_2}) \in \R^{r_1+r_2} : \sum_{i=1}^{r_1} x_i + \sum_{i=1}^{r_2} x_{r_1+i} = 0\right\},
\end{equation*}
and choose a fundamental domain $D \subseteq H$ of the unit lattice of $F$, $\mathrm{Pr}:\R^{r_1+r_2} \rightarrow H$ to be the orthogonal projection onto $H$, and define $\Log^n: \A_\infty^n \rightarrow \R^{r_1+r_2}$ by
\begin{equation*}
\Log^n(x) = (\log \|x\|_{\sigma_1}, \ldots, \log \|x\|_{\sigma_{r_1+r_2}}).
\end{equation*}
Then there exists a one-to-$w_F$ correspondence between
\begin{equation*}
\mathbb P^{n-1}(F)\mbox{ and } S(g) :=\left\{(i, x): {1 \leq i \leq h, x \in F^n \backslash \{0\}, a_ixg \mbox{ primitive}, \atop (xg)_\infty \in (\mathrm{Pr} \circ \Log^n)^{-1}(D)}\right\}.
\end{equation*}
To describe the correspondence, take $\bar z \in \mathbb P^{n-1}(F)$ and let $z= (z_1, \ldots, z_n) \in F^n$ be any representative of $\bar z$. Then there exists exactly one $1 \leq i \leq h$ such that $a_izg$ is $F$-equivalent to a primitive element of $\A_F^n$, that is, there exists an element $c \in F^*$, unique up to the units, such that $c a_i zg$ is primitive. To elaborate, for each $\nu \in \f$ let $J_\nu \subseteq F_\nu$ be the fractional ideal generated by the entries of $(zg)_\nu$, and take $J = \prod \nu^{\mathrm{ord}_\nu J_\nu}$; now choose the unique $i$ so that $I_iJ$ is principal --- in fact, $I_iJ = (c^{-1})$. Furthermore, by the definition of $D$, there exists $u \in \cO_F^*$, unique up to the roots of unity of $F$, such that the infinite part of $uca_izg$ lies in $(\mathrm{Pr} \circ \Log^n)^{-1}(D)$. We take $x = ucz$, up to the roots of unity.

Therefore, if $\phi = \phi_\f\phi_\infty$, where $\phi_\f : \A_\f^n \rightarrow \R$ is the characteristic function of the primitive vectors, and $\phi_\infty: \A_\infty^n \rightarrow \R$ is the characteristic function of the set of $x_\infty \in (\mathrm{Pr} \circ \Log^n)^{-1}(D)$ with $H_\infty(x_\infty) \leq 1$, then
\begin{equation*}
\sum_{i=1}^h \phi(a_ixg) = \begin{cases} 1 & \mbox{if $x \in S(g)$ and $H(xg) \leq 1$} \\ 0 & \mbox{otherwise.} \end{cases}
\end{equation*}
Accordingly we let
\begin{equation} \label{eq:pn-1}
f(x) = \frac{1}{w_F} \sum_{i=1}^h \phi(a_ix).
\end{equation}
One may now apply this $f$ to our formulas to study the statistics of the rational points on $\mathbb P^{n-1}$. The formula below may be useful for such purposes; see Schanuel (\cite{Scha79}) or Thunder (\cite{Thu96}) for similar computations. 

\begin{proposition} \label{prop:pn-1}
Let $\phi$ be as above. Then
\begin{equation*}
\int_{\A_F^n} \phi(x) d\alpha^n = \frac{V_n^{r_1}V_{2n}^{r_2}n^{r_1+r_2-1}2^{nr_2}R_F}{|\Delta_F|^{\frac{n}{2}}\zeta_F(n)}.
\end{equation*}
\end{proposition}
\begin{proof}
From the previous sections (see e.g. the proof of Theorem \ref{thm:prim}), it is clear that
\begin{equation*}
\int_{\A_\f^n} \phi_\f d\alpha^n = \frac{1}{\zeta_F(n)}.
\end{equation*}
It remains to concern ourselves with $\phi_\infty$. On each $F_\sigma$ for $\sigma \mid \infty$, take the polar coordinates $d\alpha^n_\sigma = r_\sigma^{n-1}dr_\sigma d\theta_\sigma$ if $\sigma$ is real, and $d\alpha^n_\sigma = 2^nr_\sigma^{2n-1}dr_\sigma d\theta_\sigma$ if $\sigma$ is complex. Let $e_\sigma = 1$ if $\sigma$ real and $e_\sigma = 2$ for $\sigma$ complex. Then
\begin{align*}
\int_{\A_\infty^n} \phi_\infty d\alpha^n &= (nV_n)^{r_1}(2nV_{2n})^{r_2}\int \prod_{\sigma\, \mathrm{real}} r_\sigma^{n-1} dr_\sigma \prod_{\sigma\, \mathrm{cplx}} r_\sigma^{2n-1} 2^ndr_\sigma \\
&= V_n^{r_1}V_{2n}^{r_2}n^{r_1+r_2}2^{nr_2}\int \prod_{\sigma \mid \infty} r_\sigma^{e_\sigma n-1}e_\sigma dr_\sigma,
\end{align*}
where the region of integration is the set of all $(r_\sigma)_{\sigma \mid \infty} \subseteq (\R_{>0})^{r_1+r_2}$ such that
\begin{equation*}
\prod_{\sigma \mid \infty} r_\sigma^{e_\sigma} \leq 1, \mbox{ and } \left( e_\sigma\log r_\sigma - \frac{e_\sigma}{d}\log\prod_{\rho \mid \infty} r_\rho^{e_\rho}\right)_{\sigma \mid \infty} \in D.
\end{equation*}
To simplify, let us take the change of coordinates $x_\sigma = e_\sigma\log r_\sigma$. Then the above integral becomes
\begin{equation*}
V_n^{r_1}V_{2n}^{r_2}n^{r_1+r_2}2^{nr_2}\int \prod_{\sigma \mid \infty} e^{nx_\sigma}dx_\sigma,
\end{equation*}
where the region of the integration is the set of all $(x_\sigma)_{\sigma \mid \infty}$ satisfying
\begin{equation*}
\sum_{\sigma \mid \infty} x_\sigma \leq 0, \mbox{ and } \left( x_\sigma - \frac{e_\sigma}{d}\sum_{\rho \mid \infty} x_\rho \right)_{\sigma \mid \infty} \in D.
\end{equation*}
We make one more change of coordinates, by putting
\begin{equation*}
y_1 = \sum_\sigma x_\sigma, y_i = x_{\sigma_i} - \frac{e_{\sigma_i}}{d}y_1 \mbox{ for } i=2, \ldots, r_1+r_2,
\end{equation*}
or equivalently,
\begin{equation*}
x_{\sigma_1} = \frac{e_{\sigma_1}}{d}y_1 - y_2 - \ldots - y_{r_1+r_2},\, x_{\sigma_i} = y_i + \frac{e_{\sigma_i}}{d}y_1 \mbox{ for } i=2, \ldots, r_1+r_2.
\end{equation*}
One computes the Jacobian matrix to be
\begin{equation*}
\left(\frac{\partial x_{\sigma_i}}{\partial y_j}\right)_{1 \leq i,j \leq r_1+r_2} =
\begin{pmatrix}
\frac{e_{\sigma_1}}{d} & -1 & \cdots & -1 \\
\frac{e_{\sigma_2}}{d} & 1 & & \\
\vdots & & \ddots & \\
\frac{e_{\sigma_{r_1+r_2}}}{d} & & & 1
\end{pmatrix},
\end{equation*}
whose determinant is equal to $1$. Therefore, we have
\begin{equation*}
\int_{\A_\infty^n} \phi_\infty d\alpha^n = V_n^{r_1}V_{2n}^{r_2}n^{r_1+r_2}2^{nr_2} \int e^{ny_1} dy_1 \ldots dy_{r_1+r_2},
\end{equation*}
where the region of the integration is given by the conditions
\begin{equation*}
y_1 \leq 0, (-y_2- \ldots -y_{r_1+r_2}, y_2, \ldots, y_{r_1+r_2}) \in D.
\end{equation*}
Therefore
\begin{equation*}
\int_{\A_\infty^n} \phi_\infty d\alpha^n =V_n^{r_1}V_{2n}^{r_2}n^{r_1+r_2-1}2^{nr_2}R_F.
\end{equation*}
This proves the proposition.
\end{proof}

Let us now estimate the second moment
\begin{equation*}
\int_{{\mathcal X}_n} \left( \sum_{x \in F^n \backslash \{0\}} f_B(xg) \right)^2 d\mu.
\end{equation*}

As in the previous section, this is equal to
\begin{equation*}
\int_{\A^{2n}} f_B(x_1) f_B(x_2) d\alpha^n(x_1)d\alpha^n(x_2) + \sum_{c \in F^*} \int_{\A^n} f_B(x)f_B(cx) d\alpha^n(x),
\end{equation*}
and the first integral is simply
\begin{equation*}
\alpha(f_B)^2 = (CB^n)^2
\end{equation*}
with
\begin{equation*}
C = \frac{V_n^{r_1}V_{2n}^{r_2}n^{r_1+r_2-1}2^{nr_2}h_FR_F}{|\Delta_F|^{\frac{n}{2}}\zeta_F(n)w_F}
\end{equation*}
by Proposition \ref{prop:pn-1}. The second sum of integrals can again be handled as in the previous section to be shown to be bounded by a constant times
\begin{align}
&\sum_q \sum_p \sum_{u \in \mathcal{O}^*_F}\int_{\mathbb{A}^n} f_{B}(x) f_{B}(\gamma ux) d\alpha^n \notag \\
&=\sum_q \sum_p \sum_{u \in \mathcal{O}^*_F}\sum_{1\leq i, j \leq h_F} \frac{B^n}{w_F^2} \int_{\mathbb{A}^n} \phi(a_i x) \phi(a_j\gamma ux) d\alpha^n, \label{eq:pn-1oink}
\end{align}
where $p,q,\gamma=\gamma(pq^{-1})$ are as in \eqref{eq:2ndmtp1}. Hence we are led to investigate
\begin{equation*}
\sum_{u \in \mathcal{O}^*_F} \int_{\mathbb{A}^n} \phi(a_i x) \phi(a_j\gamma ux) d\alpha^n = \sum_{u \in \mathcal{O}^*_F} \int_{\mathbb{A}^n} \phi(x) \phi(a_ja_i^{-1}\gamma ux) d\alpha^n
\end{equation*}
for each $1 \leq i, j \leq h_F$. Applying \eqref{eq:oink2} to the $\A_\f$ part of the integral, we find that we may bound this by
\begin{equation*}
\frac{\alpha_\f^n(\phi_\f)}{|\Delta_F|^{\frac{n}{2}}(N\tilde q)^n} \sum_{u \in \mathcal{O}^*_F} \int_{\mathbb{A}_\infty^n} \phi_\infty(x_\infty) \phi_\infty((a_ja_i^{-1}\gamma ux)_\infty) d\alpha_\infty^n,
\end{equation*}
where $\tilde q = q \cdot I_1^{-\beta_1} \cdots I_h^{-\beta_h}$ with $\beta_i = \min(1,\ord_{I_i}(q))$. Moreover, we may bound
\begin{equation*}
\sum_{u \in \cO_F^*} \phi_\infty((a_ja_i^{-1}\gamma ux)_\infty) \leq 1,
\end{equation*}
since for each $x_\infty \in \A_\infty^n$, $\phi_\infty((a_ja_i^{-1}\gamma ux)_\infty)$ is nonzero for at most one $u \in \cO_F^*$. Therefore
\begin{equation*}
\sum_{u \in \mathcal{O}^*_F} \int_{\mathbb{A}^n} \phi(x) \phi(a_ja_i^{-1}\gamma ux) d\alpha^n \leq \frac{\alpha^n(\phi)}{|\Delta_F|^{\frac{n}{2}}(N\tilde q)^n}.
\end{equation*}
We may now proceed as in the estimate of \eqref{eq:2ndmtp1} to prove that \eqref{eq:pn-1oink} is of size $O_F(\alpha_F^n(f_B))$, as desired.


\Addresses

\end{document}